\documentclass[
11pt,                          
draft,                         
english                        
]{article}

\usepackage[baseline]{euflag}
\usepackage{amsmath}           
\usepackage{amssymb}           
\usepackage[utf8]{inputenc}    
\usepackage[T1]{fontenc}       
\usepackage[a4paper]{geometry} 
\usepackage{setspace}          
\usepackage{mathtools}         
\usepackage{stmaryrd}          
\usepackage[amsmath,thmmarks,hyperref]{ntheorem} 
\usepackage{bbm}
\usepackage[activate={true,nocompatibility},
            final,tracking=true,kerning=true,
            spacing=true,factor=1100,
            stretch=10,shrink=10,expansion=false
           ]{microtype}
\microtypecontext{spacing=nonfrench}  
\definecolor{mydarkblue}{RGB}{0,0,155}
\usepackage[
           final=true,         
           colorlinks=true,    
           allcolors=mydarkblue,  
           hypertexnames=false,
           plainpages=false,              
           pdfpagelabels=true,            
           pdfencoding=auto,              
           unicode=true                   
                      ]{hyperref}         
\usepackage[shortcuts]{extdash}  

\renewcommand{\mathbb}[1]{\mathbbm{#1}}

\theoremheaderfont{\normalfont\bfseries}
\theorembodyfont{\itshape}
\newtheorem{lemma}{Lemma}

\newtheorem{proposition}[lemma]{Proposition}
\newtheorem{theorem}[lemma]{Theorem}
\newtheorem{corollary}[lemma]{Corollary}
\newtheorem{definition}[lemma]{Definition}

\theorembodyfont{\rmfamily}

\newtheorem{remark}[lemma]{Remark}

\theoremstyle{nonumberplain}
\theoremheaderfont{\normalfont\bfseries}
\theorembodyfont{\itshape}

\theoremheaderfont{\normalfont\bfseries}
\theorembodyfont{\normalfont}
\theoremseparator{:}
\theoremsymbol{\hbox{$\boxempty$}}
\newtheorem{proof}{Proof}

\makeatletter
\newcommand{\myitem}[1][argument not optional]{%
\item[#1.)]\protected@edef\@currentlabel{\textit{#1}}%
}
\makeatother

\newcommand{\argument}       {\ignorespaces{\,\cdot\,}\ignorespaces}

\DeclarePairedDelimiter{\abs}{\lvert}{\rvert}
\DeclarePairedDelimiter{\norm}{\lVert}{\rVert}

\DeclarePairedDelimiter{\ordinaryIP}{\langle}{\rangle}
\DeclarePairedDelimiter{\ordinarySet}{\{}{\}}
\DeclarePairedDelimiter{\ordinaryKom}{[}{]}

\newcommand{\RR}{\mathbb{R}}
\newcommand{\NN}{\mathbb{N}}
\newcommand{\ZZ}{\mathbb{Z}}
\newcommand{\CC}{\mathbb{C}}
\newcommand{\DD}{\mathbb{D}}
\newcommand{\set}[3][]{\ordinarySet[#1]{\,#2 \;#1|\; #3\,}}
\newcommand{\skal}[3][]{\ordinaryIP[#1]{\,#2 \,#1|\, #3\,}}
\newcommand{\ring}[1]{\mathcal{#1}}
\newcommand{\A}{\ring{A}}
\newcommand{\SOS}[1]{{\textstyle \sum\!#1^2}}
\newcommand{\D}{\mathrm{d}}
\newcommand{\Stetig}{\mathcal{C}}
\newcommand{\E}{\mathrm{e}}
\newcommand{\Motzkin}{\mathcal{M}}
\newcommand{\coefficient}[3][]{\ordinaryKom[#1]{#2}_{#3}}
\newcommand{\rad}{r}
\newcommand{\Kernel}{\mathcal K}
\newcommand{\Tail}{\mathcal T}
\newcommand{\IOp}{\mathcal I}
\newcommand{\Nexpl}{N_\mathrm{expl}}
\newcommand{\firstInternalTag}{\tag{\textup{$\clubsuit$}}}
\newcommand{\secondInternalTag}{\tag{\textup{$\spadesuit$}}}
\newcommand{\thirdInternalTag}{\tag{\textup{$\blacklozenge$}}}
\DeclareMathOperator{\ub}{ub}

\title{%
The Effective Lasserre's Perturbative Positivstellensatz
}

\author{
Igor Klep\thanks{University of Ljubljana, Faculty of Mathematics and Physics, Jadranska 21, 1000 Ljubljana \&
University of Primorska, Faculty of Mathematics, Natural Sciences and Information Technologies,
Glagoljaška 8, 6000 Koper, Slovenia. Email: \texttt{igor.klep@fmf.uni-lj.si}} \and
Victor Magron\thanks{Universit\'{e} de Toulouse; LAAS-CNRS, 7 avenue du colonel Roche, F-31400 Toulouse, France. Email: \texttt{victor.magron@laas.fr}} \and
Matthias Schötz\thanks{Universität Würzburg, Institute of Mathematics, Emil-Fischer-Straße 31, 97074 Würzburg, Germany. Email: \texttt{matthias.schoetz@uni-wuerzburg.de}}
}

\date{} 

\begin{document}

\begin{onehalfspace}

\maketitle

\begin{abstract}
We study sum-of-squares (SOS) certificates for nonnegative polynomials $p$ on $\RR^d$ and their implications for polynomial optimization over unbounded domains. Building on Lasserre’s perturbation approach, we consider SOS representations of $p$ augmented by weighted polynomial tails of the form $\sum_{n=0}^N (x\cdot x)^n/(n!)^t$ for $0<t<1$. 
Our main result provides an explicit quantitative bound on the truncation order $N$ required to achieve an $\varepsilon$-accurate certificate. Using positivity properties of the Mehler kernel and techniques inspired by polynomial kernel methods, we show that $N$ grows polynomially in $1/\varepsilon$, with rate $N = O((\|p\|/\varepsilon)^{1/(1-t)})$. 
\end{abstract}



\allowdisplaybreaks

\section{Introduction}
\subsection{Main results}
We write $\RR[x_1,\dots,x_d]$ for the algebra of real polynomials in $d \in \NN$ variables $x_1,\dots,x_d$,
and we define the convex cone of sums of squares,
\begin{equation*}
  \SOS{\RR[x_1,\dots,x_d]}
  \coloneqq
  \set[\Big]{
    \sum\nolimits_{\ell=1}^k p_\ell^2
  }{
    k\in \NN_0;\;
    p_1,\dots,p_k \in \RR[x_1,\dots,x_d]
  }.
\end{equation*}
Clearly $p(\xi) \ge 0$ for all $p \in \SOS{\RR[x_1,\dots,x_d]}$ and $\xi \in \RR^d$.
In the one-dimensional case, the converse is also true, i.e.,
a polynomial $p \in \RR[z]$ fulfills $p(\xi) \ge 0$ for all $\xi \in \RR$ if and only if $p \in \SOS{\RR[z]}$
(this is a well-known consequence of the fundamental theorem of algebra, see e.g.~\cite[Prop.~3.1]{schmuedgen:TheMomentProblem}).
But if $d\ge 2$, then there is e.g.~the Motzkin polynomial
\begin{equation*}
  \Motzkin \coloneqq x_1^2 x_2^2 (x_1^2 + x_2^2 - 3) + 1 \in \RR[x_1,\dots,x_d],
\end{equation*}
which fulfills $\Motzkin(\xi) \ge 0$ for all $\xi \in \RR^d$, but $\Motzkin + \lambda \notin \SOS{\RR[x_1,\dots,x_d]}$
for all $\lambda \in \RR$ \cite[Prop.~1.2.2 and Rem.~1.2.3]{marshall}.
However, $\SOS{\RR[x_1,\dots,x_d]}$ is in fact dense in the convex cone of pointwise positive polynomials
in the following sense:

\begin{theorem} \label{theorem:Lasserre}
  Let $p \in \RR[x_1,\dots,x_d]$, then the following are equivalent:
  \begin{enumerate}
    \myitem[i] \label{item:Lasserre:pos}
      $p(\xi) \ge 0$ for all $\xi \in \RR^d$;
    \myitem[ii] \label{item:Lasserre:perturbation}
      For every $\epsilon \in {]0,\infty[}$ there is $N\in \NN_0$ such that
      \begin{equation}
        \label{eq:Lasserre}
        p + \epsilon \sum_{n=0}^N \sum_{j=1}^d \frac{x_j^{2n}}{n!} \in \SOS{\RR[x_1,\dots,x_d]}
        ;
      \end{equation}
    \myitem[ii'] \label{item:Lasserre:perturbationStronger}
      For every $\epsilon \in {]0,\infty[}$ there is $N\in \NN_0$ such that
      \begin{equation}
        \label{eq:LasserreStronger}
        p + \epsilon \sum_{n=0}^N \sum_{j=1}^d \frac{x_j^{2n}}{(2n)!} \in \SOS{\RR[x_1,\dots,x_d]}
        .
      \end{equation}
  \end{enumerate}
\end{theorem}
The equivalence of \ref{item:Lasserre:pos} and \ref{item:Lasserre:perturbation} has been proven in \cite{lasserre}, 
by combining Carleman's condition of the theory of moments and a duality result of convex optimization.
Our first contribution is a short proof of the equivalence of \ref{item:Lasserre:pos} and \ref{item:Lasserre:perturbationStronger}
by a simple separation argument (and utilizing Carleman's condition).
While the nontrivial implication \ref{item:Lasserre:pos}~$\Rightarrow$~\ref{item:Lasserre:perturbationStronger}
is formally stronger than \ref{item:Lasserre:pos}~$\Rightarrow$~\ref{item:Lasserre:perturbation}, this improved asymptotics
of the perturbation term can in fact also be obtained by minor adaptations of the original argument in \cite{lasserre},
yet our proof is significantly shorter (see Section~\ref{section:easyProof}).

We write $\lfloor\argument \rfloor, \lceil\argument \rceil \colon \RR \to \ZZ$
for the floor and ceiling functions,
i.e., $\lfloor \lambda \rfloor \coloneqq \max \set{z\in \ZZ}{z \le \lambda}$
and $\lceil \lambda \rceil\coloneqq \min \set{z\in \ZZ}{\lambda \le z}$.
Our main result is an explicit estimate for the upper bound of the degree of the perturbation like in \eqref{eq:Lasserre} and \eqref{eq:LasserreStronger},
but for more generous perturbation terms.

\begin{theorem} \label{theorem:effective}
  Let $M \in \NN$, $p\in \RR[x_1,\dots,x_d]$ of degree at most $2M$, and $t\in {[0,1[}$, $\epsilon \in {]0,\infty[}$.
  Assume that $p(\xi) \ge 0$ for all $\xi \in \RR^d$ and set
  \begin{equation}
    \label{eq:effective:mu}
    \mu \coloneqq M\,\binom{2M+d-1}{d-1}^{\!\frac{1}{2}} 2^{M+\frac12} \in {]0,\infty[}.
  \end{equation}
  Let $\Nexpl(p/\epsilon) \in \NN_0$ be the maximum of
  \begin{equation}
    \label{eq:Nexpl:lame}
    2M + \biggl\lceil \frac{5d}{2} \biggr\rceil + 2
  \end{equation}
  and
  \begin{equation}
    \label{eq:Nexpl:relevant}
    \Bigl\lceil
    2\E \bigl( 4\mu\,\norm{p/\epsilon} \bigr)^{\frac{1}{1-t}}
    \Bigr\rceil
    +
    \max\Biggl\{
      0
      ,
      \biggl\lceil
      \frac{1}{1-t}
      \ln\biggl( \frac{ (d+1)^2\,2^{2d}\, \norm{p / \epsilon} }{ \mu } \biggr)
      \biggr\rceil
    \Biggr\}
    +
    \biggl\lfloor \frac{3d}{2} \biggr\rfloor
    +
    1
    ,
  \end{equation}
  where $\E$ is Euler's number and with the norm $\norm{\argument}$ on $\RR[x_1,\dots,x_d]$ that will be defined in
  Definition~\ref{definition:norm}. Then
  \begin{equation}
    p
    +
    \epsilon \sum_{n=0}^{\Nexpl(p/\epsilon)} \frac{1}{(n!)^t} \Bigl(\sum\nolimits_{j=1}^d x_j^2 \Bigr)^n
    \in
    \SOS{\RR[x_1,\dots,x_d]}.
  \end{equation}
\end{theorem}
Clearly $\Nexpl(p/\epsilon) \sim \norm[\big]{p/\epsilon}^{\frac{1}{1-t}}$ when $\epsilon \to 0$,
so $\Nexpl(p/\epsilon)$ has polynomial growth for perturbation terms with a $1/(n!)^t$-weight, $t \in {[0,1[}$.
Theorem~\ref{theorem:Lasserre} shows that one is also guaranteed to eventually obtain
a sum of squares by adding a $1/(n!)^t$-weighted perturbation with $t \in {[1,2]}$, but in this case
we cannot provide an estimate for the (presumably faster) growth of the required degree of the perturbation.
We also do not know whether $t=2$ is the maximal number for which adding a $1/(n!)^t$-weighted perturbation
guarantees a decomposition into a sum of squares. However, as the intersection of
$\SOS{\RR[x_1,\dots,x_d]}$ with any finite-dimensional linear subspace of $\RR[x_1,\dots,x_d]$ is closed
(see e.g.~\cite[Cor.~3.34]{laurent}) and as the Motzkin polynomial $\Motzkin$ is not a sum of squares, one can iteratively
construct a sequence $(\epsilon_n)_{n\in \NN_0}$ in ${]0,\infty[}$ such that
\begin{equation}
  \Motzkin + \sum_{n=0}^{\smash N} \epsilon_n \Bigl(\sum\nolimits_{j=1}^d x_j^2 \Bigr)^n \notin \SOS{\RR[x_1,\dots,x_d]}
  \quad\quad\text{for all $N\in \NN_0$.}
\end{equation}

The proof of Theorem~\ref{theorem:effective} is given in Section~\ref{section:quantitative}.
Our technique is inspired by the polynomial kernel method 
developed first by Fang and Fawzi \cite{fawzi} (see
also the survey \cite{OverviewConvergenceRates})
with some necessary adaptations to this non-compact case.
Specifically, we exploit the properties of an integration kernel constructed from Hermite polynomials that we use in order
to construct a sum-of-squares approximation of any pointwise positive polynomial, and then we estimate the approximation error. 

\subsection{Related works}

\subsubsection{Lasserre's perturbative Positivstellensatz}

The perturbative Positivstellensatz \cite{lasserre} by Lasserre has been generalized to other spaces of positive polynomials, by deriving adequate analogs of Carleman's condition. 
The first noncommutative analog is derived in \cite[Thm.~1]{navascues2013paradox}: the authors prove that for any element of a Weyl algebra which is nonnegative in the Schr\"odinger representation there exists another arbitrarily close element that admits a sum-of-Hermitian-squares decomposition.
Then in \cite[Thm.~B]{klep2023globally}, the authors prove that every trace-positive noncommutative polynomial admits an explicit approximation in the 1-norm on its coefficients by sums of hermitian squares and commutators of noncommutative polynomials.
Eventually, it was shown in \cite[Thm.~6.9]{klep2025sums} that moment polynomials (which are polynomial expressions in commuting variables and their formal mixed moments) that are positive on actual probability measures are sums of squares and formal moments of squares up to arbitrarily small perturbation of their coefficients.

The perturbative result not only has theoretical significance but also practical implications for some observed numerical inaccuracies in computational results. Specifically, incorrect outcomes, stemming from numerical inaccuracies, have been reported across diverse applications of semidefinite programming solvers. This issue has been particularly notable in uses of Lasserre's hierarchy for solving polynomial optimization problems (see, e.g., \cite{waki2012generate}).
An intuitive, though informal, mathematical explanation for these inaccuracies was initially proposed in \cite{lasserre} and \cite{navascues2013paradox}. 
Later, \cite{lasserre2019sdp} systematically clarified this behavior by interpreting the solver's output as a form of robust optimization, an inherent and consistent feature of its operation.
In \cite{magron2021exact}, the authors introduced a hybrid numeric-symbolic algorithm designed to compute exact sum-of-squares certificates for polynomials within the interior of the sum-of-squares cone. This approach leverages semidefinite programming solvers to obtain an approximate decomposition, by incorporating deliberate perturbations to the input polynomial's coefficients. 

\subsubsection{Degree bounds for Positivstellensätze and Lasserre's moment-SOS hierarchy}

While our present study focuses on polynomials being \emph{globally} nonnegative on $\RR^d$, there has been substantial research dedicated to analyze the asymptotic behavior of two important hierarchies \cite{lasserre2001global} in polynomial optimization on (mostly) \emph{compact} constraint sets, based on Schm\"udgen's Positivstellensatz \cite{Schmudgen1991} and Putinar's Positivstellensatz
\cite{putinar1993positive}.
Let us assume that $p \in \RR[x_1,\dots,x_d]$ is positive over a given basic compact semialgebraic set
$\set[\big]{\xi \in \RR^d}{ g_1(\xi) \geq 0, \dots, g_m(\xi) \geq 0 }$, where each $g_j \in \RR[x_1,\dots,x_d]$.
Then Schm\"udgen's Positivstellensatz \cite{Schmudgen1991} states that there exists $N \in \NN_0$ such that
$p$ lies in the truncated preordering $\mathcal{T}_{N}$ defined by
\begin{align*}
\label{eq:preorder}
\mathcal{T}_{N}
&\coloneqq
\set[\Bigg]{
  \sum_{J \subseteq \{1,\dots,m\}} \sigma_J g_J
}{
  \begin{array}{l}
    m\in \NN_0\text{ and }
    \sigma_J \in \SOS{\RR[x_1,\dots,x_d]} \text{ for }J \subseteq \{1,\dots,m\} \\
    \text{such that }\deg (\sigma_J g_J) \leq 2N
  \end{array}
}
,
\end{align*} 
where $g_J \coloneqq \prod_{j \in J} g_j$ for  $J \subseteq \{1,\dots,m\}$.
Under the additional assumption that the feasible set includes a ball constraint in its description, Putinar's Positivstellensatz \cite{putinar1993positive} states that there exists $N \in \NN_0$ such that $p$ lies in the truncated quadratic module $\mathcal{Q}_{N}$ defined by 
\begin{align*}
\label{eq:quadmodule}
\mathcal{Q}_{N}
&\coloneqq
\set[\Bigg]{\sum_{j=0}^m \sigma_j g_j}{
  \begin{array}{l}
    m\in \NN_0\text{ and }
    \sigma_j \in \SOS{\RR[x_1,\dots,x_d]} \text{ for }j \in \{1,\dots,m\} \\
    \text{such that }\deg (\sigma_j g_j) \leq 2N
  \end{array}
}
.
\end{align*} 
For a general compact basic  semialgebraic set, Schm\"udgen-type hierarchies were shown in \cite{schweighofer2004complexity} to certify the nonnegativity of $p + \varepsilon$ at degree bounds of  $N=O(\varepsilon^{-c})$, where $c$ is a positive constant that depends on the set. 
Initially, Putinar-type hierarchies were known from \cite{nie2007complexity} to certify nonnegativity more slowly, with degree bounds of $N=O(\exp(\varepsilon^{-c}))$. 
Recent work in \cite{baldi2023effective} significantly improved this latter result, demonstrating that Putinar-type hierarchies also certify nonnegativity at $N=O(\varepsilon^{-c})$, matching the performance of Schm\"udgen-type hierarchies.
For special cases of Schm\"udgen-type hierarchies, including the unit sphere \cite{fawzi}, the unit ball and standard simplex \cite{slot2022sum}, and the hypercube \cite{laurent2023effective}, it has been established that one can select $c = 1/2$. 
Similar degree bounds have also been obtained in \cite{rateprod} for feasible sets that are Cartesian products of such sets (unit sphere, unit ball, or standard simplex). 
The binary hypercube $\{0,1\}^d$ is addressed separately in \cite{slot2023sum}, while degree bounds of $O(1/\varepsilon)$ for Putinar-type hierarchies over the hypercube were recently established in \cite{baldi2024degree}.
In the context of polynomial optimization with correlative sparsity, better degree bounds have been derived in \cite{korda2025convergence}, assuming sufficiently sparse input data. 
For \emph{non-compact} feasible sets one can derive Putinar-type hierarchies based on sums of squares of
\emph{rational functions} with uniform denominators. In this case, similar degree bounds of
$N=O(\varepsilon^{-c})$ have been found in \cite{mai2022complexity}.
For a detailed overview of these results, including the tightness of performance analysis, exponential convergence under local optimality conditions, and finite convergence, we refer to the recent survey \cite{OverviewConvergenceRates}.\looseness=-1

\subsection{Acknowledgments} 
This work has been supported by European Union’s HORIZON–MSCA-2023-DN-JD programme under the Horizon Europe (HORIZON) Marie Skłodowska-Curie Actions, grant agreement 101120296 (TENORS), the project COMPUTE, funded within the QuantERA II Programme that has received funding from the EU's H2020 research and innovation programme under the GA No 101017733 {\normalsize\euflag}.
IK also acknowledges support of the Slovenian Research Agency program P1-0222 and grants J1-50002, N1-0217, J1-60011, J1-50001, J1-3004 and J1-60025. 
Partially supported by the Fondation de l’\'Ecole polytechnique
as part of the Gaspard Monge Visiting Professor Program. IK and MS thank \'Ecole polytechnique and Inria
for hospitality during the preparation of this manuscript.

\section{General facts and notation} \label{sec:notation}
We will use standard multi-index notation: For $\alpha \in \NN_0^d$ we write
$\alpha! \coloneqq \prod_{j=1}^d \alpha_j !$ and $\abs{\alpha} \coloneqq \sum_{j=1}^d \alpha_j$,
and then $x^\alpha \coloneqq \prod_{j=1}^d x_j^{\alpha_j}$. Multi-indices can be compared pointwise,
i.e., $\alpha \le \beta$ for $\alpha,\beta \in \NN_0^d$ means that $\alpha_j \le \beta_j$ for all $j\in\{1,\dots,d\}$.
In this case we also write $\binom{\beta}{\alpha} \coloneqq \prod_{j=1}^d \binom{\beta_j}{\alpha_j}$.
The floor and ceiling functions $\lfloor\argument \rfloor, \lceil\argument \rceil \colon \RR^d \to \ZZ^d$ are defined
entrywise.

We view polynomial algebras $\RR[\argument]$ as subalgebras of the corresponding algebras of formal power series $\RR[[\argument]]$.
For $p \in \RR[[x_1,\dots,x_d]]$ and $\alpha \in \NN_0^d$ we denote by $\coefficient{p}{\alpha} \in \RR$ the coefficient
of $x^\alpha$ so that $p = \sum_{\alpha\in \NN_0^d} \coefficient{p}{\alpha} \, x^\alpha$.
Let $D \in \NN_0$, then we write
\begin{equation}
  \RR[x_1,\dots,x_d]_{D}
  \coloneqq
  \set[\big]{p\in \RR[x_1,\dots,x_d]}{\coefficient{p}{\alpha} = 0\text{ for all $\alpha \in \NN_0^d$ with $\abs{\alpha} > D$}}
\end{equation}
for the space of polynomials in $x_1,\dots,x_d$ of degree at most $D$.

The convex cone $\SOS{\RR[x_1,\dots,x_d]}$ induces a partial order $\preccurlyeq$ on $\RR[x_1,\dots,x_d]$,
namely, let $p,q \in \RR[x_1,\dots,x_d]$, then\!
\begin{equation}
  p \preccurlyeq q\quad\quad\text{if and only if}\quad\quad q-p \in \SOS{\RR[x_1,\dots,x_d]}.
\end{equation}
Clearly $p \preccurlyeq q$ implies $p(\xi) \le q(\xi)$ for all $\xi \in \RR^d$,
and in the univariate case (i.e., $d=1$) the converse is also true,
see e.g.~\cite[Prop.~3.1]{schmuedgen:TheMomentProblem}. Note that
\begin{equation}
  0\preccurlyeq p,
  \quad\quad
  pq \preccurlyeq pr,
  \quad\quad{and}\quad\quad
  q+s \preccurlyeq r+s
\end{equation}
for all $p \in \SOS{\RR[x_1,\dots,x_d]}$ and $q,r,s \in \RR[x_1,\dots,x_d]$ with $q \preccurlyeq r$.
Next, consider $p, p' \in \RR[x_1,\dots,x_d]$ and $q,q' \in \SOS{\RR[x_1,\dots,x_d]}$:
\begin{align}
  \label{eq:productOfBounds}
  \text{If}\quad{-q} \preccurlyeq p &\preccurlyeq q\quad\text{and}\quad{-q'} \preccurlyeq p' \preccurlyeq q',
  \quad\text{then}\quad{-qq'} \preccurlyeq pp' \preccurlyeq qq',
\shortintertext{because}
  qq'+pp' &= \frac{(q+p)(q'+p') + (q-p)(q'-p')}{2} \in \SOS{\RR[x_1,\dots,x_d]} \nonumber
\shortintertext{and}
  qq'-pp' &= \frac{(q+p)(q'-p') + (q-p)(q'+p')}{2} \in \SOS{\RR[x_1,\dots,x_d]} \nonumber
  .
\end{align}
Moreover, the following estimate will also be helpful:
\begin{equation}
  \label{eq:pseudoCS}
  - \frac12\biggl( \lambda p^2 + \frac{q^2}{\lambda} \biggr)
  \preccurlyeq
  pq
  \preccurlyeq
  \frac12\biggl( \lambda p^2 + \frac{q^2}{\lambda} \biggr)
  \quad\quad\text{for any $p,q \in \RR[x_1,\dots,x_d]$ and $\lambda \in {]0,\infty[}$,}
\end{equation}
this follows from expanding
$\bigl(\lambda^{1/2} p \pm \lambda^{-1/2} q \bigr)^2 \in \SOS{\RR[x_1,\dots,x_d]}$.

\section{A separation argument} \label{section:easyProof}
The goal of this section is to give a short proof of Theorem~\ref{theorem:Lasserre}.
For $a,b \in \RR[x_1,\dots,x_d]$ we define the \emph{order interval}
$[a,b] \coloneqq \set[\big]{p\in \RR[x_1,\dots,x_d]}{a \preccurlyeq p \preccurlyeq b}$,
which is clearly a convex subset of $\RR[x_1,\dots,x_d]$.
Moreover, $[-a,a]$ is balanced for all $a\in \RR[x_1,\dots,x_d]$, i.e., $-p \in [-a,a]$ for all $p\in [-a,a]$.
We also define
\begin{equation}
  U
  \coloneqq
  \bigcup_{N\in \NN_0}
  \biggl[
    \,- \sum_{n=0}^N \sum_{j=1}^d \frac{x_j^{2n}}{(2n)!}\,{ , }\,\sum_{n=0}^N \sum_{j=1}^d \frac{x_j^{2n}}{(2n)!}\,
  \biggr]
  \subseteq
  \RR[x_1,\dots,x_d]
  .
\end{equation}
Being a union of an increasing sequence of balanced convex subsets of $\RR[x_1,\dots,x_d]$, $U$ is again balanced and convex.

\begin{lemma} \label{lemma:absorbing}
  The subset $U$ of $\RR[x_1,\dots,x_d]$ is absorbing, i.e., for every $p \in \RR[x_1,\dots,x_d]$
  there is $\lambda \in {[0,\infty[}$ such that $p \in \lambda U$.
\end{lemma}
\begin{proof}
  Write $V \coloneqq \bigcup_{\lambda \in {[0,\infty[}} \lambda U$, then we have to show that $V = \RR[x_1,\dots,x_d]$.
  As $U$ is balanced and convex, $V$ is a linear subspace of $\RR[x_1,\dots,x_d]$.
  Therefore it is sufficient to show that $x^\alpha \in V$ for all monomials $x^\alpha$ with $\alpha \in \NN_0^d$.
  It is already clear that $x_j^{2n} \in V$ for all $j \in \{1,\dots,d\}$ and $n\in \NN_0$.
  We also note the following consequence of the definition of $U$
  as a union of an increasing sequence of balanced order intervals:
  if $p \in \RR[x_1,\dots,x_d]$ and $q \in \SOS{\RR[x_1,\dots,x_d]} \cap V$ fulfill
  $-q \preccurlyeq p \preccurlyeq q$, then $p \in V$.

  Consider $\alpha \in \NN_0^d$ such that $\alpha_j = 0$ for all $j \in \{2, \dots, d\}$, then
  \begin{equation*}
    - \frac12\bigl( 1+x_1^{2\alpha_1} \bigr)
    \preccurlyeq
    x^\alpha
    \preccurlyeq
    \frac12\bigl( 1+x_1^{2\alpha_1} \bigr)
  \end{equation*}
  by \eqref{eq:pseudoCS},
  and therefore $x^\alpha \in V$ because $\frac12\bigl( 1+x_1^{2\alpha_1} \bigr) \in \SOS{\RR[x_1,\dots,x_d]} \cap V$.
  Now assume that there is $i \in \{2, \dots,d\}$ such that $x^\alpha \in V$
  for all $\alpha \in \NN_0^d$ that fulfill $\alpha_j = 0$ for all $j \in \{i,\dots,d\}$.
  Consider $\alpha \in \NN_0^d$ such that $\alpha_j = 0$ for all $j \in \{i+1,\dots,d\}$,
  and define $\beta \in \NN_0^d$ as $\beta_j \coloneqq \alpha_j$ for $j \in \{1,\dots,i-1\}$ and $\beta_j \coloneqq 0$ for $j \in \{i,\dots,d\}$,
  so $x^{2\beta} \in \SOS{\RR[x_1,\dots,x_d]} \cap V$. Then $x^\alpha = x^\beta x_i^{\alpha_i}$ and
  \begin{equation*}
    - \frac12 \bigl( x^{2\beta} + x_i^{2\alpha_i} \bigr)
    \preccurlyeq
    x^\alpha
    \preccurlyeq
    \frac12 \bigl( x^{2\beta} + x_i^{2\alpha_i} \bigr)
  \end{equation*}
  by \eqref{eq:pseudoCS}.
  This shows that $x^\alpha \in V$ because $\frac12\bigl( x^{2\beta} + x_i^{2\alpha_i} \bigr) \in \SOS{\RR[x_1,\dots,x_d]} \cap V$.
  By induction it follows that $x^\alpha \in V$ for all $\alpha \in \NN_0^d$.
\end{proof}
To summarize the discussion so far, $U$ is an absorbing balanced convex subset of $\RR[x_1,\dots,x_d]$.
Its Minkowski functional $\norm{\argument}_U \colon \RR[x_1,\dots,x_d] \to {[0,\infty[}$,
\begin{equation}
  p \mapsto \norm{p}_U \coloneqq \inf \set[\big]{\lambda \in {[0,\infty[}}{ p \in \lambda U}
\end{equation}
therefore is a well-defined seminorm on $\RR[x_1,\dots,x_d]$.

\begin{remark}
  The Minkowski functional $\norm{\argument}_U$ is actually a norm on $\RR[x_1,\dots,x_d]$: if $p \in \RR[x_1,\dots,x_d]$
  fulfills $\norm{p}_U = 0$, then $p \in \epsilon U$ for all $\epsilon \in {]0,\infty[}$, which means that there are $N(\epsilon) \in \NN_0$ such that
  \begin{equation*}
    - \epsilon \sum_{n=0}^{N(\epsilon)} \sum_{j=1}^d \frac{x_j^{2n}}{(2n)!}
    \preccurlyeq p \preccurlyeq
    \epsilon \sum_{n=0}^{N(\epsilon)} \sum_{j=1}^d \frac{x_j^{2n}}{(2n)!}
    \,,\quad\text{so}\quad
    \abs{p(\xi)} \le \epsilon \sum_{n=0}^{N(\epsilon)} \sum_{j=1}^d \frac{\xi_j^{2n}}{(2n)!} \le \epsilon \sum_{j=1}^d \cosh(\xi_j)
    \quad\text{for }\xi\in\RR^d.
  \end{equation*}
  As this is true for all $\epsilon \in {]0,\infty[}$ and all $\xi \in \RR^d$, it follows that $p=0$.
  However, for the proof of Theorem~\ref{theorem:Lasserre}, this fact that $\norm{\argument}_U$ is a norm will not be relevant.
\end{remark}
We say that a linear functional $\omega \colon \RR[x_1,\dots,x_d] \to \RR$ is \emph{positive} if $\omega(p^2) \ge 0$
for all $p\in \RR[x_1,\dots,x_d]$, or equivalently, $\omega(q) \ge 0$ for all $q\in \SOS{\RR[x_1,\dots,x_d]}$.
A fundamental result by Nussbaum from the theory of moments asserts: if a positive linear functional $\omega$ on $\RR[x_1,\dots,x_d]$
fulfills the \emph{multivariate Carleman condition}
\begin{equation}
  \label{eq:carleman}
  \sum_{n=1}^\infty \frac{1}{ \omega\bigl(x_j^{2n}\bigr)^{1/(2n)}} = \infty
  \quad\quad\textup{for all }j\in\{1,\dots,d\},
\end{equation}
where $1/0 \coloneqq \infty$, then there exists a unique positive Radon measure $\mu_\omega$ on $\RR^d$ such that
\begin{equation}
  \label{eq:integral}
  \omega(p) = \int_{\RR^d} p(\xi)\,\D\mu_\omega(\xi)
  \quad\quad\textup{for all }p\in\RR[x_1,\dots,x_d].
\end{equation}
For details we refer to e.g.~the textbook \cite[Thm.~14.19]{schmuedgen:TheMomentProblem}.

\begin{proof}[of Theorem~\ref{theorem:Lasserre}]
  Let $p \in \RR[x_1,\dots,x_d]$. First assume that $\epsilon \in {]0,\infty[}$ and $N \in \NN_0$ fulfill
  $p+\epsilon \sum_{n=0}^N \sum_{j=1}^d x_j^{2n} / (2n)! \in \SOS{\RR[x_1,\dots,x_d]}$. Then
  \begin{eqnarray*}
    0
    \preccurlyeq
    p+\epsilon \sum_{n=0}^{\smash N} \sum_{j=1}^{\smash d} \frac{x_j^{2n}}{(2n)!}
    \preccurlyeq
    p+\epsilon \sum_{n=0}^{\smash N} \sum_{j=1}^{\smash d} \frac{x_j^{2n}}{n!}
  \end{eqnarray*}
  holds,
  which proves the implication \ref{item:Lasserre:perturbationStronger}~$\Rightarrow$~\ref{item:Lasserre:perturbation}.

  Next assume that for every $\epsilon \in {]0,\infty[}$ there is $N \in \NN_0$ such that
  $p+\epsilon \sum_{n=0}^N \sum_{j=1}^d x_j^{2n} / n! \in \SOS{\RR[x_1,\dots,x_d]}$. Consider any
  $\xi \in \RR^d$. Then
  \begin{equation*}
    0
    \le
    p(\xi)+\epsilon \sum_{n=0}^{\smash{N(\epsilon)}} \sum_{j=1}^{\smash d} \frac{\xi_j^{2n}}{n!}
    \le
    p(\xi)+\epsilon \sum_{j=1}^d \exp(\xi_j^2)
  \end{equation*}
  holds for all $\epsilon \in {]0,\infty[}$, with suitable $N(\epsilon) \in \NN_0$ depending on $\epsilon$,
  and therefore $0 \le p(\xi)$.
  This proves the implication \ref{item:Lasserre:perturbation}~$\Rightarrow$~\ref{item:Lasserre:pos}.

  Finally, assume there is $\epsilon \in {]0,\infty[}$ such that
  $p+\epsilon \sum_{n=0}^N \sum_{j=1}^d x_j^{2n}/ (2n)! \notin \SOS{\RR[x_1,\dots,x_d]}$ for all $N \in \NN_0$.
  Note that this assumption implies $(p+\epsilon U) \cap \SOS{\RR[x_1,\dots,x_d]} = \emptyset$:
  indeed, if there would exist $q \in (p+\epsilon U) \cap \SOS{\RR[x_1,\dots,x_d]}$, then $q-p \in \epsilon U$,
  so in particular there would exist $N \in \NN_0$ such that $q-p \preccurlyeq \epsilon \sum_{n=0}^N \sum_{j=1}^d x_j^{2n}/ (2n)!$,
  and therefore $0 \preccurlyeq q \preccurlyeq p + \epsilon \sum_{n=0}^N \sum_{j=1}^d x_j^{2n}/ (2n)!$, a contradiction.
  In order to prove the remaining implication \ref{item:Lasserre:pos}~$\Rightarrow$~\ref{item:Lasserre:perturbationStronger}
  we will show that there is $\xi \in \RR^d$ such that $p(\xi) < 0$. By definition of the (semi-)norm $\norm{\argument}_U$,
  the set $U$ contains the open $1$-ball of $\norm{\argument}_U$, so $p$ is not contained in the $\norm{\argument}_U$-closure
  of $\SOS{\RR[x_1,\dots,x_d]}$ because $(p+\epsilon U) \cap \SOS{\RR[x_1,\dots,x_d]} = \emptyset$.
  By the Hahn--Banach theorem \cite[Cor.~3.3.9]{davidson},
  there exists a $\norm{\argument}_U$-bounded linear functional
  $\omega \colon \RR[x_1,\dots,x_d] \to \RR$ such that $\omega(q) > \omega(p)$ for all $q \in \SOS{\RR[x_1,\dots,x_d]}$.
  From $0 \in \SOS{\RR[x_1,\dots,x_d]}$ it follows that $\omega(p)  < 0$, and consequently $\omega(q) \ge 0$ for all
  $q \in \SOS{\RR[x_1,\dots,x_d]}$ because $\SOS{\RR[x_1,\dots,x_d]}$ is closed under multiplication with scalars $\lambda \in {[0,\infty[}$.
  So $\omega$ is a positive linear functional on $\RR[x_1,\dots,x_d]$ such that $\omega(p) < 0$.
  As $\omega$ is $\norm{\argument}_U$-bounded, there is $\mu \in {]0,\infty[}$ such that $\abs{\omega(u)} \le \mu$
  for all $u \in U$. In particular $\sum_{n=0}^N \sum_{j=1}^d \omega(x_j^{2n}) / (2n)! \le \mu$ for all $N\in \NN_0$,
  and therefore $\omega(x_j^{2n}) \le \mu (2n)!$ for all $n\in \NN_0$ and $j\in \{1,\dots,d\}$.
  It follows that $\omega$ fulfills the multivariate Carleman condition \eqref{eq:carleman}:
  \begin{equation*}
    \sum_{n=1}^\infty \frac{1}{\omega\bigl( x_j^{2n} \bigr)^{1/(2n)}}
    \ge
    \sum_{n=1}^\infty \frac{1}{\bigl( \mu (2n)! \bigr)^{1/(2n)}}
    \ge
    \sum_{n=1}^\infty \frac{1}{\mu^{1/(2n)} 2n }
    \ge
    \frac{1}{2 \max \{1, \sqrt{\mu}\}}
    \sum_{n=1}^\infty \frac{1}{ n }
    =
    \infty
  \end{equation*}
  for all $j\in \{1,\dots,d\}$.
  Therefore there is a unique positive Radon measure $\mu_\omega$ on $\RR^d$ that represents the positive functional $\omega$
  as in \eqref{eq:integral}. Consequently, as $\omega(p) < 0$, there necessarily exists a point $\xi \in \RR^d$
  such that $p(\xi) < 0$.
\end{proof}

\begin{remark}
  As an application of Theorem~\ref{theorem:Lasserre} we
  show that $\norm{\argument}_U$ is a weighted $L^\infty$-norm:
  \begin{equation}
    \label{eq:normInterpretation}
    \norm{p}_U = \sup_{\xi \in \RR^d} \frac{\abs{p(\xi)}}{\sum_{j=1}^d \cosh(\xi_j) }
    \quad\quad\text{for all $p\in\RR[x_1,\dots,x_d]$.}
  \end{equation}
  To check that this is true it suffices to show that the open unit ball of the norm $\norm{\argument}_U$
  on the left-hand side of \eqref{eq:normInterpretation} is contained in the closed unit ball of the norm
  defined by the right-hand side of \eqref{eq:normInterpretation}, and vice versa.
  First assume $p \in \RR[x_1,\dots,x_d]$ fulfills $\norm{p}_U < 1$. Then $p\in U$
  and consequently there is $N\in \NN_0$ such that
  \begin{equation*}
    \abs[\big]{p(\xi)}
    \le
    \sum_{n=0}^N \sum_{j=1}^d \frac{\xi_j^{2n}}{(2n)!}
    \le
    \sum_{j=1}^d \cosh(\xi_j)
  \end{equation*}
  for all $\xi \in \RR^d$. Conversely, assume $p \in \RR[x_1,\dots,x_d]$ and $\epsilon \in {]0,1[}$ fulfill
  $\abs{p(\xi)} \le (1-\epsilon) \sum_{j=1}^d \cosh(\xi_j)$ for all $\xi \in \RR^d$. We show that this implies $p\in U$,
  hence $\norm{p}_U\le 1$: For every $N\in \NN_0$ write
  \begin{eqnarray*}
    P(N) \coloneqq \set[\Big]{\xi \in \RR^d}{ \abs{p(\xi)} \ge (1-\epsilon/2) \sum\nolimits_{n=0}^N \sum\nolimits_{j=1}^d \xi_j^{2n} / (2n)! }
    .
  \end{eqnarray*}
  For $N' \in \NN_0$ strictly greater than half the degree of $p$, the set $P(N')$ is compact.
  Moreover, for every $\xi \in \RR^d$ there is $N\in \NN_0$ such that $(1-\epsilon) \sum_{j=1}^d \cosh(\xi_j) < (1-\epsilon/2) \sum\nolimits_{n=0}^N \sum\nolimits_{j=1}^d \xi_j^{2n} / (2n)!$, and therefore
  $\abs{p(\xi)} < (1-\epsilon/2) \sum\nolimits_{n=0}^N \sum\nolimits_{j=1}^d \xi_j^{2n} / (2n)!$,
  i.e., $\xi \notin P(N)$. This shows that $\bigcap_{N\in \NN_0} P(N) = \emptyset$,
  and therefore the complements of all the closed sets $P(N)$ with $N\in \NN_0$ cover $P(N')$.
  By compactness of $P(N')$ and as $P(N) \supseteq P(N+1)$ for all $N\in \NN_0$,
  there is $N_0\in \NN_0$ such that $\RR^d \setminus P(N_0) \supseteq P(N')$,
  and without loss of generality we can assume that $N_0 \ge N'$. It follows that
  $P(N_0) \cap P(N') = \emptyset$ and $P(N_0) \subseteq P(N')$, so $P(N_0) = \emptyset$.
  This means
  \begin{equation*}
    \abs[\big]{p(\xi)}
    <
    \biggl(1-\frac\epsilon2\biggr) \sum_{n=0}^{\smash{N_0}} \sum_{j=1}^{\smash d} \frac{\xi_j^{2n}}{ (2n)! }
    \quad\quad
    \text{for all $\xi \in \RR^d$.}
  \end{equation*}
  Applying Theorem~\ref{theorem:Lasserre} to the two pointwise (strictly) positive polynomials
  \begin{equation*}
    q_\pm
    \coloneqq
    \biggl(1-\frac\epsilon2\biggr) \sum_{n=0}^{\smash{N_0}} \sum_{j=1}^{\smash d} \frac{x_j^{2n}}{ (2n)! } \pm p
    \in
    \RR[x_1,\dots,x_d]
  \end{equation*}
  shows that there are two numbers $N_\pm \in \NN_0$ such that
  \begin{equation*}
    q_\pm + \frac\epsilon2 \sum_{n=0}^{\smash{N_\pm}} \sum_{j=1}^{\smash d} \frac{x_j^{2n}}{ (2n)! } \in \SOS{\RR[x_1,\dots,x_d]}
    .
  \end{equation*}
  For $N_{\max} \coloneqq \max\{N_0,N_+,N_-\}$ it follows that
  \begin{equation*}
    - \sum_{n=0}^{\smash {N_{\max}}} \sum_{j=1}^{\smash d} \frac{x_j^{2n}}{ (2n)! }
    \preccurlyeq
    p
    \preccurlyeq
    \sum_{n=0}^{\smash {N_{\max}}} \sum_{j=1}^{\smash d} \frac{x_j^{2n}}{ (2n)! }
  \end{equation*}
  and therefore $p \in U$.
\end{remark}

\section{A positive integration kernel} \label{section:quantitative}
In this section we will derive the explicit estimates of Theorem~\ref{theorem:effective} for the degree $\Nexpl$ of the
perturbation. Our approach relies on Hermite polynomials in $d$ variables $x_1, \dots, x_d$,
see e.g.~\cite{szego:ortoPoly}, \cite{rainville}, or \cite{askey}.

\subsection{Hermite polynomials and the truncated Mehler kernel}
One way to define the Hermite polynomials is via the explicit formula
\begin{align}
  \label{eq:HermiteExplicit1D}
  &&
  H_n
  \coloneqq{}&
  n! \sum_{k=0}^{\smash{\lfloor n/2 \rfloor}} \frac{(-1)^k (2z)^{n-2k}}{k!\,(n-2k)!} \in \RR[z]
  &&\text{for $n\in \NN_0$,}
\shortintertext{and then}
  \label{eq:HermiteExplicitMultiD}
  &&H_\alpha
  \coloneqq
  \prod_{j=1}^d H_{\alpha_j}(x_j)
  ={}&
  \alpha! \sum_{\substack{\beta \in \NN_0^d \\ 2\beta \le \alpha}} \frac{(-1)^{\abs{\beta}} 2^{\abs{\alpha-2\beta}} x^{\alpha-2\beta}}{\beta!\,(\alpha-2\beta)!}
  \in \RR[x_1,\dots,x_d]
  &&\text{for $\alpha\in \NN_0^d$.}
\end{align}
It is straightforward to check that
\begin{equation}
  \label{eq:HermiteRecursive}
  H_0 = 1,\quad H_1 = 2z,\quad \text{and}\quad H_{n+1} =  2z H_n - 2n H_{n-1}\quad\text{for all $n\in \NN$}.
\end{equation}
Recall the Gaussian integral for $\ell \in \NN_0$ and $\lambda \in {]0,\infty[}$:
\begin{align}
  \label{eq:GaussianIntegral}
  &&
  \int_\RR \xi^\ell \exp\biggl(-\frac{\xi^2}{\lambda} \biggr)\,\frac{\D\xi}{\sqrt{\lambda\pi}}
  &=
  \mathrlap{
    \begin{cases}
      0 & \text{if $\ell $ is odd,} \\
      \frac{\lambda^{\ell / 2}  \ell ! }{ 2^\ell (\ell /2)! } & \text{if $\ell $ is even.}
    \end{cases}
  }
\shortintertext{Therefore}
  \label{eq:GaussianIntegralEstimateEven}
  &&
  \int_\RR \xi^{2m} \exp\biggl(-\frac{\xi^2}{\lambda} \biggr)\,\frac{\D\xi}{\sqrt{\lambda\pi}}
  &\le
  \lambda^m m!
  &&\text{for all $m\in \NN_0$, $\lambda \in {]0,\infty[}$,}
\shortintertext{and}
  \label{eq:GaussianIntegralEstimate}
  &&
  \int_\RR \xi^\ell \exp\biggl(-\frac{\xi^2}{\lambda} \biggr)\,\frac{\D\xi}{\sqrt{\lambda\pi}}
  &\le
  \frac{\lambda^{\ell/2} \sqrt{\ell!}}{2^{\ell/2}}
  &&\text{for all $\ell\in \NN_0$, $\lambda \in {]0,\infty[}$,}
\end{align}
because, for $\ell = 2m$, the right-hand side of \eqref{eq:GaussianIntegral} is
$\lambda^m \ell! / (2^\ell m!)$,
and the estimates $\ell! / (2^\ell m!) = 2^{-\ell} \binom{\ell}{m} m! \le m!$
and $\ell! / (2^\ell m!) = \sqrt{\ell!}\,2^{-\ell} \binom{\ell}{m}^{1/2} \le \sqrt{\ell!}\, 2^{-\ell/2}$
hold.
Using the recursive relation \eqref{eq:HermiteRecursive} one easily checks that
\begin{align}
  \int_\RR H_n(\xi)\,\xi^m \,\exp\bigl(-\xi^2 \bigr)\,\frac{\D\xi}{\sqrt{\pi}}
  &=
  0
\intertext{for $n\in \NN_0$, $m\in \{0,\dots,n-1\}$, and}
  \label{eq:superfluous}
  \int_\RR H_n(\xi)\,\xi^{n+\ell} \,\exp\bigl(-\xi^2 \bigr)\,\frac{\D\xi}{\sqrt{\pi}}
  &=
  \begin{cases}
    0 & \text{if $\ell $ is odd,} \\
    \frac{(n +\ell) ! }{ 2^\ell (\ell /2)! } & \text{if $\ell $ is even}
  \end{cases}
\intertext{for $n,\ell\in \NN_0$. In particular}
  \label{eq:HermiteOrtho}
  \int_\RR H_n(\xi)\,H_m(\xi) \,\exp\bigl(-\xi^2 \bigr)\,\frac{\D\xi}{\sqrt{\pi}}
  &=
  0
\shortintertext{and}
  \int_\RR H_n(\xi)^2 \,\exp\bigl(-\xi^2 \bigr)\,\frac{\D\xi}{\sqrt{\pi}}
  =
  \int_\RR H_n(\xi)\,(2\xi)^n \,&\exp\bigl(-\xi^2 \bigr)\,\frac{\D\xi}{\sqrt{\pi}}
  =
  2^n n!
  \label{eq:HermiteNorm}
\end{align}
for $n\in \NN_0$, $m\in \{0,\dots,n-1\}$.

In order to efficiently deal with the multi-dimensional case we use the dot-product
for $\ell$-tuples $r = (r_1,\dots,r_\ell), s = (s_1,\dots,s_\ell) \in \ring{R}^\ell$ in any ring $\ring{R}$, namely
\begin{equation}
  (r\cdot s) \coloneqq \sum_{j=1}^{\smash\ell} r_j s_j \in \ring{R}.
\end{equation}
Note that by the multinomial theorem,
\begin{equation}
  \label{eq:multinomial}
  (r\cdot s)^n = \sum_{\substack {\alpha \in \NN_0^\ell \\ \abs\alpha = n}} \frac{r^\alpha s^\alpha}{\alpha!}
  \quad\quad\text{for all }n\in \NN_0.
\end{equation}
We will use the notation of the dot-product e.g.~in the special case $\ell = d$ and $\ring{R} = \RR[x_1,\dots,x_d]$,
then $(x \cdot \xi) = \sum_{j=1}^d x_j \xi_j \in \RR[x_1,\dots,x_d]$ for all $\xi \in \RR^d$.

We write $B^d(\rad)$ for the closed euclidean ball in $\RR^d$ with radius $\rad \in {[0,\infty[}$, and
we write $\Stetig\bigl(B^d(\rad)\bigr)$ for the vector space of continuous real-valued functions on it.
We then define the inner product
$\skal{\argument}{\argument}_\rad \colon \Stetig\bigl(B^d(\rad)\bigr) \times \Stetig\bigl(B^d(\rad)\bigr) \to \RR$,
\begin{equation}
  (f,g)
  \mapsto
  \skal{f}{g}_\rad \coloneqq \int_{B^d(\rad)} f(\xi)\,g(\xi)\,\exp\bigl(-(\xi\cdot\xi)\bigr)\,\frac{\D^d\xi}{\sqrt{\pi}^d}
\end{equation}
and for $f,g \in \Stetig(\RR^d)$ we define
\begin{equation}
  \label{eq:skalInftyDef}
  \skal{f}{g}_\infty \coloneqq \lim_{\rad \to \infty}  \skal{f}{g}_\rad
  \quad\quad\text{(provided the limit exists).}
\end{equation}
Note that
\begin{equation}
  \label{eq:monotoneInnerproduct}
  \skal{f}{f}_{\rad} \le \skal{f}{f}_{\rad'}  \le \skal{f}{f}_\infty
\end{equation}
for all $f \in \Stetig(\RR^d)$ and all $\rad,\rad' \in {[0,\infty[}$ with $\rad\leq\rad'$ (and possibly $\skal{f}{f}_\infty = \infty$).
We identify $\RR[x_1,\dots,x_d]$ with the linear subspace of polynomial functions in $\Stetig(\RR^d)$.
From \eqref{eq:HermiteOrtho} and \eqref{eq:HermiteNorm} it then follows that
\begin{equation}
  \label{eq:HermiteOrthoMultiD}
  \skal[\big]{H_\alpha}{H_\beta}_\infty
  =
  \prod_{j=1}^{\smash d} \int_{\RR} H_{\alpha_j}(\xi_j)\,H_{\beta_j}(\xi_j)\,\exp\bigl(-\xi_j^2\bigr)\,\frac{\D\xi_j}{\sqrt{\pi}}
  =
  \delta_{\alpha,\beta} 2^{\abs{\alpha}} \alpha!
  \quad\quad\text{for all $\alpha,\beta\in \NN_0^d$.}
\end{equation}
Consider any $D\in \NN_0$. The multi-dimensional Hermite polynomials $H_\alpha$ with $\alpha \in \NN_0^d$, $\abs{\alpha} \le D$
clearly form a basis of $\RR[x_1,\dots,x_d]_{D}$. Consequently, if $p\in \RR[x_1,\dots,x_d]_{D}$, then
\begin{equation}
  \label{eq:orthobasis}
  p
  =
  \sum_{\substack{\alpha\in\NN_0^d \\ \abs{\alpha} \le D}}
  \frac{H_\alpha}{2^{\abs\alpha} \alpha!} \skal{H_\alpha}{p}_\infty
  \quad\quad\text{and}\quad\quad
  \skal{H_\alpha}{p}_\infty = 0
  \quad\text{for $\alpha \in \NN_0^d$ with $\abs{\alpha}>D$.}
\end{equation}
Moreover, by combining \eqref{eq:GaussianIntegralEstimateEven} and \eqref{eq:monotoneInnerproduct} we obtain the estimate
\begin{equation}
  \label{eq:monomialInnerproducts}
  \skal{x^\alpha}{x^\alpha}_\rad
  \le
  \skal{x^\alpha}{x^\alpha}_\infty
  \le
  \alpha!
\end{equation}
for all $\rad \in {[0,\infty]}$ and $\alpha \in \NN_0^d$.

The formal sum of products of Hermite polynomials in different variables is described by Mehler's formula.
Note that for every real algebra $\A$ and every $p\in \A[[\nu]]$ the formal power series
$\exp(\nu p) \in \A[[\nu]]$ is well-defined, as in every degree of $\nu$ only finitely
many terms contribute. Mehler's formula in the univariate case (i.e., $d=1$) states that the identity
\begin{equation}
  \label{eq:Mehler1D}
  \sum_{n=0}^\infty \frac{\nu^n H_n(x) H_n(y)}{2^n n!}
  =
  \bigl(1-\nu^2\bigr)^{-\frac{1}{2}}
  \exp\biggl( - \frac{x^2\nu^2 -2xy\nu + y^2\nu^2}{1-\nu^2} \biggr)
\end{equation}
holds in $\RR[[x,y,\nu]]$ (see, e.g., \cite[Prob.~23, p.~380]{szego:ortoPoly} or
\cite[\S111, p.~198]{rainville} or \cite[(6.1.13)]{askey}).
It follows that
\begin{equation}
  \label{eq:MehlerMultiD}
  \sum_{\alpha \in \NN_0^d} \frac{\nu^{\abs{\alpha}} H_\alpha(x) H_\alpha(y)}{2^{\abs{\alpha}} \alpha!}
  =
  \bigl(1-\nu^2\bigr)^{-\frac{d}{2}}
  \exp\biggl( - \frac{(x\cdot x) \nu^2 - 2(x\cdot y) \nu + (y\cdot y) \nu^2}{1-\nu^2} \biggr)
\end{equation}
holds in $\RR[[x_1,\dots,x_n,y_1,\dots,y_n,\nu]]$.
Note that this formal power series is absolutely and
locally uniformly convergent on $\CC^d \times \CC^d \times \DD$, where $\DD$ is the open unit disc in $\CC$.

We will construct a positive integration kernel by cutting the exponential series in \eqref{eq:MehlerMultiD}
at a sufficiently large even power. For $N\in \NN_0$ and $\lambda \in {[0,1[}$ we first define
\begin{equation}
  \label{eq:expNDef}
  \exp_N
  \coloneqq
  \sum_{n=0}^{2N} \frac{z^n}{n!} \in \RR[z] .
\end{equation}
and then we define the \emph{truncated Mehler kernel} $\Kernel_{N,\lambda} \colon \RR^d \to \RR[x_1,\dots,x_d]_{4N}$,
\begin{align}
  \label{eq:KNDef}
  \xi \mapsto \Kernel_{N,\lambda}(\xi)
  \coloneqq
  \bigl(1-\lambda^2\bigr)^{-\frac d 2}
  \exp_N\biggl(-\frac{(x\cdot x) \lambda^2 - 2(x\cdot \xi) \lambda}{1-\lambda^2} \biggr)
  \exp\biggl( - \frac{(\xi \cdot \xi) \lambda^2}{1-\lambda^2} \biggr)
  .
\end{align}
All component functions $\coefficient{ \Kernel_{N,\lambda}}{\alpha} \colon \RR^d \to \RR$
for $\alpha \in \NN_0^d$ are products of a polynomial and a Gaussian factor;
in particular they are bounded from above and below by a constant function.
We also define, for all $N \in \NN_0$, $\lambda \in {[0,1[}$, and $\rad \in {[0,\infty]}$,
the integral operator $\IOp_{N,\lambda,\rad} \colon \RR[x_1,\dots,x_d] \to \RR[x_1,\dots,x_d]_{4N}$,
\begin{equation}
  \label{eq:IOpDef}
  p
  \mapsto
  \IOp_{N,\lambda,\rad}(p)
  \coloneqq
  \sum_{\substack{\alpha \in \NN_0^d \\ \abs{\alpha} \le 4N}}
  \skal[\big]{ \coefficient{\Kernel_{N,\lambda}}{\alpha} }{ p }_\rad \, x^\alpha
  .
\end{equation}
Note that $\IOp_{N,\lambda,\rad}$ is indeed also defined for $\rad = \infty$ because the functions
$\abs[\big]{\coefficient{\Kernel_{N,\lambda}}{\alpha}}$ are bounded on $\RR^d$.

\begin{lemma} \label{lemma:expN}
  For all $N \in \NN_0$ the polynomial $\exp_N$ is a sum of squares, i.e., $\exp_N \in \SOS{\RR[z]}$.
\end{lemma}
\begin{proof}
  We only have to check that $\exp_N(\xi) \ge 0$ for all $\xi \in \RR$, $N\in \NN_0$, so let any
  $N\in \NN_0$ and $\xi \in \RR$ be given.
  If $\xi \ge 0$, then it is immediately clear from \eqref{eq:expNDef} that $\exp_N(\xi) > 0$.
  If $\xi < 0$, then by the Lagrange form of the remainder for the series expansion of $\exp$
  there exists $\eta \in {[\xi, 0]}$ such that
  \begin{align*}
    \exp(\xi)
    &=
    \exp_N(\xi) + \frac{\exp({\eta})}{(2N+1)!}\,\xi^{2N+1},
  \shortintertext{and therefore}
    \exp_N(\xi)
    &=
    \exp(\xi) - \frac{\exp({\eta})}{(2N+1)!}\,\xi^{2N+1} > 0
  \end{align*}
  because $\xi^{2N+1} < 0$.
\end{proof}

\begin{proposition} \label{proposition:SOS}
  Let $N \in \NN_0$, $\lambda \in {[0,1[}$, $\rad \in {[0,\infty]}$, and $p\in \RR[x_1,\dots,x_d]$.
  Assume that $p(\xi) \ge 0$ for all $\xi \in \RR^d$.
  Then $\IOp_{N,\lambda,\rad}(p) \in \SOS{\RR[x_1,\dots,x_d]}\cap \RR[x_1,\dots,x_d]_{4N}$.
\end{proposition}
\begin{proof}
  Write $C$ for the convex cone $\SOS{\RR[x_1,\dots,x_d]}\cap \RR[x_1,\dots,x_d]_{4N}$.
  From the previous Lemma~\ref{lemma:expN} it follows that $\Kernel_{N,\lambda}(\xi)\,p(\xi) \in C$ for all $\xi \in \RR^d$.
  As $\RR^d \ni \xi \mapsto \Kernel_{N,\lambda}(\xi)\,p(\xi) \in C \subseteq \RR[x_1,\dots,x_d]_{4N}$
  is continuous, we can evaluate $\IOp_{N,\lambda,\rad}(p)$ for $\rad < \infty$ as a vector-valued Riemann integral:
  \begin{equation*}
    \IOp_{N,\lambda,\rad}(p)
    =
    \sum_{\substack{\alpha \in \NN_0^d \\ \abs{\alpha} \le 4N}}
    \skal[\big]{ \coefficient{\Kernel_{N,\lambda}}{\alpha} }{ p }_r \, x^\alpha
    =
    \int_{B^d(\rad)}
    \Kernel_{N,\lambda}(\xi)\,p(\xi)\,\exp\bigl(-(\xi\cdot\xi)\bigr)\,\frac{\D^d\xi}{\sqrt{\pi}^d}
    .
  \end{equation*}
  In particular, $\IOp_{N,\lambda,\rad}(p)$ is the limit of a sequence in $C$ (namely, the Riemann sums).
  As the convex cone $C$ is closed in $\RR[x_1,\dots,x_d]_{4N}$ (see e.g.~\cite[Cor.~3.34]{laurent}) it follows that $\IOp_{N,\lambda,\rad}(p) \in C$
  for all $\rad \in {[0,\infty[}$, and then also for $\rad = \infty$.
\end{proof}

\begin{proposition} \label{proposition:rad2infty}
  Let $N \in \NN_0$, $\lambda \in {[0,1[}$, $\rad \in {[0,\infty]}$, and $p\in \RR[x_1,\dots,x_d]$.
  Then for all $\gamma \in \NN_0^d$ with $\abs{\gamma} \le 2N$ the identity
  \begin{equation}
    \label{eq:rad2infty}
    \coefficient[\big]{\IOp_{N,\lambda,\rad}(p)}{\gamma}
    =
    \sum_{ \alpha\in \NN_0^d }
    \frac{ \lambda^{\abs \alpha} \,\coefficient{H_\alpha}{\gamma}\,\skal{H_\alpha}{p}_\rad }{2^{\abs \alpha} \alpha!}
  \end{equation}
  holds, and the series on the right-hand side is absolutely convergent.
\end{proposition}
\begin{proof}
  Let $\gamma \in \NN_0^d$ with $\abs{\gamma} \le 2N$ and note that
  \begin{equation}
    \coefficient[\bigg]{\exp_N\biggl(- \frac{(x\cdot x)\lambda^2 - 2(x\cdot \xi) \lambda}{1-\lambda^2}\biggr)}{\gamma}
    =
    \coefficient[\bigg]{\exp\biggl(- \frac{(x\cdot x)\lambda^2 - 2(x\cdot \xi) \lambda}{1-\lambda^2}\biggr)}{\gamma}
    \label{eq:rad2infty:internal}
    \thirdInternalTag
  \end{equation}
  because $\coefficient[\big]{ \bigl((x\cdot x)\lambda^2 - 2(x\cdot \xi) \lambda \bigr)^n }{\gamma} = 0$
  for all $n\in \NN_0$ with $n>2N$.
  In Mehler's formula \eqref{eq:MehlerMultiD}, the coefficient of $x^\gamma$ is a formal power series in $y_1,\dots,y_d,\nu$
  that converges absolutely and uniformly on the compact space $B^d(\rad) \times [-\lambda,\lambda]$ for $r < \infty$.
  Therefore we can exchange summation and integration so that the identity
  \begin{align*}
    \coefficient[\big]{&\IOp_{N,\lambda,\rad}(p)}{\gamma}
    =
    \\
    &=
    \skal[\big]{\coefficient{\Kernel_{N,\lambda}}{\gamma}}{p}_\rad
    \\
    &\stackrel{\!\!\eqref{eq:rad2infty:internal}\!\!}{=}
    \int_{B^d(\rad)}
    \bigl(1-\lambda^2\bigr)^{-\frac d 2}
    \coefficient[\bigg]{ \exp\biggl(- \frac{(x\cdot x)\lambda^2 - 2(x\cdot \xi) \lambda}{1-\lambda^2}\biggr) }{\gamma}
    \exp\biggl( - \frac{(\xi \cdot \xi) \lambda^2}{1-\lambda^2} \biggr)
    \,p(\xi)\,\exp\bigl(-(\xi\cdot\xi)\bigr)\,\frac{\D^d\xi}{\sqrt{\pi}^d}
    \\
    &=
    \int_{B^d(\rad)}
    \sum\nolimits_{\alpha \in \NN_0^d} \frac{\lambda^{\abs{\alpha}} \,\coefficient{H_\alpha(x)}{\gamma}\, H_\alpha(\xi)}{2^{\abs{\alpha}} \alpha!}
    \,p(\xi)\,\exp\bigl(-(\xi\cdot\xi)\bigr)\,\frac{\D^d\xi}{\sqrt{\pi}^d}
    \\
    &=
    \sum_{\alpha\in\NN_0^d} \frac{\lambda^{\abs{\alpha}} \coefficient{ H_\alpha }{\gamma} }{2^{\abs \alpha} \alpha!}
    \int_{B^d(\rad)}
    H_\alpha(\xi)
    \,p(\xi)\,\exp\bigl(-(\xi\cdot\xi)\bigr)\,\frac{\D^d\xi}{\sqrt{\pi}^d}
    \\
    &=
    \sum_{\alpha\in\NN_0^d} \frac{\lambda^{\abs{\alpha}} \,\coefficient{ H_\alpha }{\gamma} \,\skal{ H_\alpha }{ p }_\rad }{2^{\abs \alpha} \alpha!}
  \end{align*}
  holds and this series over $\alpha \in \NN_0^d$ is absolutely convergent.

  Yet in order to pass to the limit $\rad \to \infty$
  we need a uniform estimate for the absolute convergence of the series on the right-hand side for all $r \in {[0,\infty]}$:
  As a first step, the estimate
  \begin{align*}
    \abs[\bigg]{ \frac{\lambda^{\abs \alpha} \,\coefficient{H_\alpha}{\gamma} \,\skal{H_\alpha}{p}_\rad }{2^{\abs \alpha} \alpha!}}
    &\le
    \frac{\lambda^{\abs \alpha} \,\abs{\coefficient{H_\alpha}{\gamma}}}{2^{\abs \alpha} \alpha!} \,\skal[\big]{H_\alpha}{H_\alpha}^{1/2}_\rad
    \skal[\big]{p}{p}^{1/2}_\rad
    \\
    &\le
    \frac{\lambda^{\abs \alpha} \,\abs{\coefficient{H_\alpha}{\gamma}}}{2^{\abs \alpha} \alpha!} \,\skal[\big]{H_\alpha}{H_\alpha}^{1/2}_\infty
    \skal[\big]{p}{p}^{1/2}_\infty
    \\
    &\le
    \frac{\lambda^{\abs \alpha} \,\abs{\coefficient{H_\alpha}{\gamma}}}{2^{\abs \alpha/2} \sqrt{\alpha!}}
    \skal[\big]{p}{p}^{1/2}_\infty
  \end{align*}
  holds for all $r\in {[0,\infty]}$,
  where we first applied the Cauchy--Schwartz inequality, and then \eqref{eq:monotoneInnerproduct} and \eqref{eq:HermiteOrthoMultiD}.
  Moreover, by \eqref{eq:HermiteExplicitMultiD},
  $\abs{\coefficient{H_\alpha}{\gamma}} = 0$ if $\alpha - \gamma \not\in 2\NN_0^d$ and, if $\alpha - \gamma \in 2\NN_0^d$, then
  \begin{align*}
    \abs{\coefficient{H_\alpha}{\gamma}}
    =
    \alpha!\,\frac{2^{\abs \gamma} }{((\alpha-\gamma)/2)! \,\gamma!}
    =
    \sqrt{\alpha!} \,\frac{2^{\abs \gamma}}{\sqrt{\gamma!}}
    \binom{\alpha-\gamma}{(\alpha-\gamma)/2}^{\frac12}
    \binom{\alpha}{\gamma}^{\frac12}
    \le
    \sqrt{\alpha!}\,\frac{2^{\abs\alpha/2 + \abs \gamma/2}}{\sqrt{\gamma!}}
    \binom{\alpha}{\gamma}^{\frac12}
    .
  \end{align*}
  So
  \begin{equation*}
    \sum_{\alpha\in \NN_0^d}
    \frac{\lambda^{\abs \alpha} \abs{\coefficient{H_\alpha}{\gamma}}}{2^{\abs \alpha/2} \sqrt{\alpha!}}
    \le
    \frac{2^{\abs \gamma/2}}{\sqrt{\gamma!}}
    \sum_{\substack{\alpha\in \NN_0^d \\ \gamma\le \alpha}}
    \lambda^{\abs \alpha} \binom{\alpha}{\gamma}^{\frac12}
    <
    \infty
  \end{equation*}
  because $\sum_{n=k}^\infty \lambda^n \binom{n}{k}^{1/2} \le (k!)^{-1/2} \sum_{n=k}^\infty \lambda^n \,n^{k/2}  < \infty$
  for all $k\in \NN_0$.
  Taken together, these estimates show that
  the right-hand side of \eqref{eq:rad2infty} is absolutely convergent also for $\rad = \infty$,
  and that for every $\epsilon \in {]0,\infty[}$ there is $n\in \NN_0$ such that
  \begin{equation*}
    \sum_{\substack{\alpha\in \NN_0^d \\ \abs{\alpha} > n}}
    \abs[\bigg]{ \frac{\lambda^{\abs \alpha} \,\coefficient{H_\alpha}{\gamma} \,\skal{H_\alpha}{p}_\rad}{2^{\abs \alpha} \alpha!} }
    \le
    \sum_{\substack{\alpha\in \NN_0^d \\ \abs{\alpha} > n}}
    \frac{\lambda^{\abs \alpha} \,\abs{\coefficient{H_\alpha}{\gamma}}}{2^{\abs \alpha/2} \sqrt{\alpha!}}
    \skal[\big]{p}{p}^{1/2}_\infty
    \le
    \epsilon/3
  \end{equation*}
  for all $\rad \in {[0,\infty]}$. It follows that
  \begin{equation*}
    \lim_{r\to\infty}
    \sum_{\alpha\in \NN_0^d}
    \frac{\lambda^{\abs \alpha} \,\coefficient{H_\alpha}{\gamma} \,\skal{H_\alpha}{p}_\rad}{2^{\abs \alpha} \alpha!}
    =
    \sum_{\alpha\in \NN_0^d}
    \frac{\lambda^{\abs \alpha} \,\coefficient{H_\alpha}{\gamma}\,\skal{H_\alpha}{p}_\infty}{2^{\abs \alpha} \alpha!}
  \end{equation*}
  because
  \begin{align*}
    \abs[\bigg]{
      \sum\nolimits_{\alpha\in \NN_0^d}
      \frac{\lambda^{\abs \alpha} \,\coefficient{H_\alpha}{\gamma}\,\skal{H_\alpha}{p}_\infty}{2^{\abs \alpha} \alpha!}
      &-
      \sum\nolimits_{\alpha\in \NN_0^d}
      \frac{\lambda^{\abs \alpha} \,\coefficient{H_\alpha}{\gamma}\,\skal{H_\alpha}{p}_\rad}{2^{\abs \alpha} \alpha!}
    }
    \le
    \\
    &\le
    \frac{2\epsilon}{3}
    +
    \abs[\bigg]{
      \sum\nolimits_{\substack{\alpha\in \NN_0^d \\ \abs{\alpha} \le n}}
      \frac{\lambda^{\abs \alpha} \,\coefficient{H_\alpha}{\gamma}}{2^{\abs \alpha} \alpha!}
      \Bigl( \skal[\big]{H_\alpha}{p}_\infty - \skal[\big]{H_\alpha}{p}_\rad \Bigr)
    }
    \\
    &\le
    \epsilon
  \end{align*}
  for all sufficiently large $\rad$.
\end{proof}

\begin{remark}
  In the case $\rad = \infty$, the above proof is rather tedious. Yet it seems there is no direct and
  less technical way to obtain this result: it cannot be obtained directly from the absolute
  convergence of the power series in the Mehler formula on $\CC^d\times\CC^d\times \DD$,
  because the asymptotic behavior in $\xi$ of its absolute values is $\sim \exp\bigl( (\xi\cdot \xi) \lambda^2 / (1-\lambda^2) \bigr)$,
  which is not integrable over $\exp\bigl(-(\xi\cdot\xi)\bigr)\,\D^d\xi$ for $\lambda$ close to $1$.
  The only reason to introduce the inner products $\skal{\argument}{\argument}_\rad$ with $\rad < \infty$
  was to use them as a tool to regularize the $\rad = \infty$\,-case in the above proof.
\end{remark}

\begin{corollary} \label{corollary:rad2infty}
  Let $M \in \NN_0$ and $p\in \RR[x_1,\dots,x_d]_{2M}$.
  Let $\lambda \in {[0,1[}$ and $N \in \NN_0$ such that $M \le N$. Then
  \begin{equation}
    \IOp_{N,\lambda,\infty}(p)
    =
    p
    +
    \sum_{\substack{\alpha\in \NN_0^d \\ \abs{\alpha}\le 2M}}
    \frac{ (\lambda^{\abs \alpha}-1) \,H_\alpha\,\skal{H_\alpha}{p}_\infty }{2^{\abs \alpha} \alpha!}
    \,+
    \sum_{\substack{\alpha\in \NN_0^d \\ 2N < \abs{\alpha}\le 4N}}
    \skal[\big]{ \coefficient{\Kernel_{N,\lambda}}{\alpha} }{ p }_\infty \, x^\alpha
    .
  \end{equation}
\end{corollary}
\begin{proof}
  By Proposition~\ref{proposition:rad2infty} and the defining identity \eqref{eq:IOpDef} for $\IOp_{N,\lambda,\infty}(p)$,
  \begin{align*}
    \IOp_{N,\lambda,\infty}(p)
    &=
    \sum_{\substack{\gamma\in \NN_0^d \\ \abs{\gamma}\le 4N}}
    \coefficient[\big]{\IOp_{N,\lambda,\infty}(p)}{\gamma}\,x^\gamma
    \\
    &=
    \sum_{\substack{\gamma\in \NN_0^d \\ \abs{\gamma}\le 2N}}
    \sum_{ \alpha\in \NN_0^d }
    \frac{ \lambda^{\abs \alpha} \,\coefficient{H_\alpha}{\gamma}\,x^\gamma\,\skal{H_\alpha}{p}_\infty }{2^{\abs \alpha} \alpha!}
    \,+
    \sum_{\substack{\gamma\in \NN_0^d \\ 2N < \abs{\gamma}\le 4N}}
    \skal[\big]{ \coefficient{\Kernel_{N,\lambda}}{\gamma} }{ p }_\infty \, x^\gamma
    .
  \end{align*}
  By \eqref{eq:orthobasis}, $p = \sum_{\alpha\in\NN_0^d, \abs{\alpha} \le 2M} H_\alpha \,\skal{H_\alpha}{p}_\infty/ ( 2^{\abs\alpha} \alpha! )$
  and $\skal{H_\alpha}{p}_\infty = 0$ for $\alpha\in \NN_0^d$ with $\abs{\alpha} > 2M$.
  Moreover, $H_\alpha \in \RR[x_1,\dots,x_d]_{2M} \subseteq \RR[x_1,\dots,x_d]_{2N}$ for all $\alpha \in \NN_0^d$ with $\abs{\alpha} \le 2M$, and therefore
  \begin{align*}
    \sum_{\substack{\gamma\in \NN_0^d \\ \abs{\gamma}\le 2N}}
    \sum_{ \alpha\in \NN_0^d }
    \frac{ \lambda^{\abs \alpha} \,\coefficient{H_\alpha}{\gamma}\,x^\gamma\,\skal{H_\alpha}{p}_\infty }{2^{\abs \alpha} \alpha!}
    &=
    \sum_{\substack{\gamma\in \NN_0^d \\ \abs{\gamma}\le 2N}}
    \sum_{\substack{\alpha\in \NN_0^d \\ \abs{\alpha}\le 2M}}
    \frac{ \lambda^{\abs \alpha} \,\coefficient{H_\alpha}{\gamma}\,x^\gamma\,\skal{H_\alpha}{p}_\infty }{2^{\abs \alpha} \alpha!}
    \\
    &=
    \sum_{\substack{\alpha\in \NN_0^d \\ \abs{\alpha}\le 2M}}
    \frac{ \lambda^{\abs \alpha} \,H_\alpha\,\skal{H_\alpha}{p}_\infty }{2^{\abs \alpha} \alpha!}
    \\
    &=
    p
    +
    \sum_{\substack{\alpha\in \NN_0^d \\ \abs{\alpha}\le 2M}}
    \frac{ (\lambda^{\abs \alpha}-1) \,H_\alpha\,\skal{H_\alpha}{p}_\infty }{2^{\abs \alpha} \alpha!}
    .
  \end{align*}
\end{proof}

Consider $M \in \NN_0$ and $p\in \RR[x_1,\dots,x_d]_{2M}$. Assume that $p(\xi) \ge 0$ for all $\xi \in \RR^d$
and let $t\in {[0,1[}$. Proposition~\ref{proposition:SOS} and Corollary~\ref{corollary:rad2infty} together show that
\begin{equation}
  0
  \preccurlyeq
  \IOp_{N,\lambda,\infty}(p)
  =
  p
  +
  \sum_{\substack{\alpha\in \NN_0^d \\ \abs{\alpha}\le 2M}}
  \frac{ (\lambda^{\abs \alpha}-1) \,H_\alpha\,\skal{H_\alpha}{p}_\infty }{2^{\abs \alpha} \alpha!}
  \,+
  \sum_{\substack{\alpha\in \NN_0^d \\ 2N < \abs{\alpha}\le 4N}}
  \skal[\big]{ \coefficient{\Kernel_{N,\lambda}}{\alpha} }{ p }_\infty \, x^\alpha
\end{equation}
for all $\lambda \in {[0,1[}$ and sufficiently large $N\in \NN_0$. In order to show that
$0 \preccurlyeq p + \sum_{n=0}^{\Nexpl(p)} (x\cdot x)^n / (n!)^t$ we will show in Section~\ref{sec:firstEstimate} that
\begin{equation}
  \label{eq:firstEstimate}
  \sum_{\substack{\alpha\in \NN_0^d \\ \abs{\alpha} \le 2M}}
  \frac{(\lambda^{\abs \alpha}-1)\, H_\alpha \, \skal[\big]{H_\alpha}{p}_\infty}{2^{\abs \alpha} \alpha!}
  \preccurlyeq
  \sum_{n=0}^{M+\lfloor d / 2\rfloor} \frac{ (x\cdot x)^n }{ (n!)^t }
\end{equation}
for $\lambda \in {[0,1[}$ sufficiently close to $1$, and then in Section~\ref{sec:secondEstimate} that
\begin{equation}
  \label{eq:secondEstimate}
  \sum_{\substack{\alpha\in \NN_0^d \\ 2N < \abs{\alpha}\le 4N}}
  \skal[\big]{ \coefficient{\Kernel_{N,\lambda}}{\alpha} }{ p }_\infty \, x^\alpha
  \preccurlyeq
  \sum_{n=N-\lfloor d/2 \rfloor}^{2N+ \lceil d/2 \rceil} \frac{ (x\cdot x)^n }{ (n!)^t }
\end{equation}
for sufficiently large $N\in \NN_0$. We will obtain explicit estimates for these parameters $\lambda$ and $N$,
which then prove Theorem~\ref{theorem:effective} for $\epsilon = 1$ (the case of general $\epsilon$ follows by rescaling).

\begin{remark}
  Heuristically, we can expect such estimates \eqref{eq:firstEstimate} and \eqref{eq:secondEstimate} to exist:
  Expanding $p$ in the basis of Hermite polynomials shows that
  the left-hand side of \eqref{eq:firstEstimate} is a diagonal linear operator on $\RR[x_1,\dots,x_d]$ with eigenvalues
  $\lambda^\alpha-1$. As $\Kernel_{N,\lambda,\infty}$ converges for $N\to \infty$ against the full Mehler kernel,
  the higher-order error-terms that are described by the left-hand side of \eqref{eq:secondEstimate} should approach zero;
  more precisely, we will see that the dominating term is the integral of $(x\cdot \xi)^n / n!$ over a Gaussian measure,
  which grows like $\sim 2^{-n/2}(x\cdot x)^{n/2} \sqrt{n!} / n! = 2^{-n/2}/\sqrt{n!}$, see \eqref{eq:GaussianIntegralEstimate}.
  The reason for treating the $(\xi\cdot \xi)\lambda^2$-component in the truncated Mehler kernel \eqref{eq:KNDef} differently
  than the $(x\cdot x)\lambda^2 - 2(x\cdot \xi) \lambda$-component is to make this second estimate work; we can integrate
  the whole Gaussian factor $\exp\bigl(-(\xi\cdot\xi) \lambda^2 / (1-\lambda^2)\bigr)$ even though integrating the individual
  parts $(\xi \cdot \xi)^n / n!$ of its series-expansion can result in a not absolutely convergent series.
\end{remark}

The bounds for admissible $\lambda$ and $N$ in these two estimates \eqref{eq:firstEstimate} and \eqref{eq:secondEstimate}
necessarily depend on $p$. We will express this dependency by utilizing a norm on polynomials that is easy to compute from the coefficients:

\begin{definition} \label{definition:norm}
  We define the norm $\norm{\argument} \colon \RR[x_1,\dots,x_d] \to {[0,\infty[}$,
  \begin{equation*}
    p \mapsto \norm{p} \coloneqq \sum_{\alpha \in \NN_0^d} \abs[\big]{ \coefficient{p}{\alpha} } (\alpha !)^{1/2}
    .
  \end{equation*}
\end{definition}
Even though $\norm{\argument}$ is not the norm induced by one of the inner products $\skal{\argument}{\argument}_\rad$
with $\rad \in {[0,\infty]}$, we can still prove the following:

\begin{proposition} \label{proposition:innerproductAndNorm}
  For all $r\in {[0,\infty]}$ and all $p \in \RR[x_1,\dots,x_d]$ the estimate
  \begin{equation}
     \label{eq:innerproductAndNorm}
    \skal{p}{p}^{1/2}_\rad
    \le
    \norm{p}
  \end{equation}
  holds.
\end{proposition}
\begin{proof}
  Indeed,
  $\skal{p}{p}_\rad^{1/2}
  \le
  \sum_{\alpha \in\NN_0^d} \abs[\big]{\coefficient{p}{\alpha}} \skal{x^\alpha}{x^\alpha}^{1/2}_\rad
  \le
  \sum_{\alpha \in\NN_0^d} \abs[\big]{\coefficient{p}{\alpha}} (\alpha!)^{1/2}
  =
  \norm{p}$,
  where the first inequality holds because the map $\RR[x_1,\dots,x_d] \ni p \mapsto \skal{p}{p}_\rad^{1/2} \in {[0,\infty[}$ is a seminorm,
  and the second inequality follows from \eqref{eq:monomialInnerproducts}.
\end{proof}

\subsection{The first estimate} \label{sec:firstEstimate}
Recall that
\begin{equation}
  \label{eq:theMaximalBinomial}
  \binom{k}{\ell} \le \binom{k}{\lfloor k/2 \rfloor}
  \quad\quad\text{for all $k\in \NN_0$ and $\ell \in \{0,\dots,k\}$,}
\end{equation}
which follows from the observation that $\binom{k}{\ell} / \binom{k}{\ell-1} = (k-\ell+1)/ \ell > 1$
if and only if $2\ell \le k$, or equivalently, $\ell \le \lfloor k/2 \rfloor$.
For these maximal binomial coefficients there is a simple formula:

\begin{lemma} \label{lemma:maximalBinomial}
  Let $n\in \NN_0$, then
  \begin{equation}
    \label{eq:maximalBinomial}
    \binom{n}{\lfloor n/2 \rfloor}
    =
    2^n \prod_{j=1}^{\lceil n/2 \rceil} \frac{2j-1}{2j}
    .
  \end{equation}
\end{lemma}
\begin{proof}
  \eqref{eq:maximalBinomial} certainly holds for $n=0$, and if \eqref{eq:maximalBinomial} holds for one $n\in \NN_0$,
  then also for $n+1$ in place of $n$. To see this, we treat the cases of odd and even $n$ separately:
  If $n$ is even, i.e., $n=2k$ for some $k \in \NN_0$, then
  \begin{equation*}
    \binom{n+1}{\lfloor (n+1)/2 \rfloor}
    =
    \binom{2k+1}{k}
    =
    \binom{2k}{k} \frac{2k+1}{k+1}
    =
    \biggl( 2^{2k} \prod\nolimits_{j=1}^k\frac{2j-1}{2j} \biggr) \frac{2k+1}{k+1}
    =
    2^{2k+1} \prod_{j=1}^{k+1}\frac{2j-1}{2j}
    .
  \end{equation*}
  If $n$ is odd, i.e., $n=2k+1$ for some $k \in \NN_0$, then
  \begin{equation*}
    \binom{n+1}{\lfloor (n+1)/2 \rfloor}
    =
    \binom{2k+2}{k+1}
    =
    \binom{2k+1}{k} \frac{2k+2}{k+1}
    =
    \biggl( 2^{2k+1} \prod\nolimits_{j=1}^{k+1}\frac{2j-1}{2j} \biggr) \frac{2k+2}{k+1}
    =
    2^{2k+2} \prod_{j=1}^{k+1}\frac{2j-1}{2j}
    .
  \end{equation*}
  By induction \eqref{eq:maximalBinomial} holds for all $n\in \NN_0$.
\end{proof}

\begin{proposition} \label{proposition:maximalBinomials}
  Suppose $n\in \NN_0$ and $k,\ell \in \{0,\dots,n\}$. Then
  \begin{equation}
    \binom{n}{k} \le 2^{n-\ell} \binom{\ell}{\lfloor \ell / 2 \rfloor}
    .
  \end{equation}
\end{proposition}
\begin{proof}
  This follows from the estimate
  \begin{equation*}
    \binom{n}{k}
    \le
    \binom{n}{\lfloor n/2 \rfloor}
    =
    2^n \prod_{j=1}^{\lceil n/2 \rceil} \frac{2j-1}{2j}
    =
    2^{n-\ell} \binom{\ell}{\lfloor \ell/2 \rfloor} \prod_{j=\lceil \ell/2 \rceil+1}^{\lceil n/2 \rceil} \frac{2j-1}{2j}
    \le
    2^{n-\ell} \binom{\ell}{\lfloor \ell/2 \rfloor}
  \end{equation*}
  where we first apply \eqref{eq:theMaximalBinomial} and then the previous Lemma~\ref{lemma:maximalBinomial}.
\end{proof}

\begin{proposition} \label{proposition:HermiteEstimates}
  The following estimate holds:
  \begin{equation}
    \label{eq:HermiteEstimates}
    -2^{\abs{\alpha}} \sqrt{\alpha !} \sum_{\substack{\beta \in \NN_0^d \\ \beta \le \lceil \alpha / 2 \rceil}} \frac{x^{2\beta}}{\beta !}
    \preccurlyeq
    H_\alpha
    \preccurlyeq
    2^{\abs{\alpha}} \sqrt{\alpha !} \sum_{\substack{\beta \in \NN_0^d \\ \beta \le \lceil \alpha / 2 \rceil}} \frac{x^{2\beta}}{\beta !}
    \quad\quad\text{for all $\alpha \in \NN_0^d$.}
  \end{equation}
\end{proposition}
\begin{proof}
  By \eqref{eq:productOfBounds} it is sufficient to treat the univariate case only, i.e.,
  \begin{equation*}
    - 2^n \sqrt{n!} \sum_{k=0}^{\smash{\lceil n/2 \rceil}} \frac{z^{2k}}{k!}
    \preccurlyeq
    H_n
    \preccurlyeq
    2^n \sqrt{n!} \sum_{k=0}^{\smash{\lceil n/2 \rceil}} \frac{z^{2k}}{k!}
    \quad\quad\text{for all $n\in \NN_0$,}
  \end{equation*}
  where $\preccurlyeq$ is the partial order on $\RR[z]$ induced by the convex cone $\SOS{\RR[z]}$,
  or equivalently, by the convex cone of pointwise positive polynomials in $z$. This means we only have to show that
  \begin{equation*}
    \abs[\big]{H_n(\xi)} \le 2^n \sqrt{n!} \sum_{k=0}^{\smash{\lceil n/2 \rceil}} \frac{\xi^{2k}}{k!}
    \quad\quad\text{for all $\xi \in \RR$, $n\in \NN_0$.}
  \end{equation*}
  So let $n\in \NN_0$ and $\xi \in \RR$ be arbitrary. Then
  \begin{equation*}
    \abs[\big]{H_n(\xi)}
    \le
    n! \sum_{\ell=0}^{\smash{\lfloor n/2 \rfloor}} \frac{ \abs{2\xi}^{n-2\ell}}{\ell!\,(n-2\ell)!}
    =
    2^n \sqrt{n!} \sum_{\ell=0}^{\smash{\lfloor n/2 \rfloor}} \frac{ \sqrt{n!}\,\abs{\xi}^{n-2\ell} }{ 2^{2\ell}\, \ell! \,(n-2\ell)!}
  \end{equation*}
  by the definition \eqref{eq:HermiteExplicit1D} of $H_n$, so the proof is complete once we have shown that
  \begin{equation*}
    \sum_{\ell=0}^{\smash{\lfloor n/2 \rfloor}} \frac{ \sqrt{n!}\,\abs{\xi}^{n-2\ell} }{ 2^{2\ell}\, \ell! \,(n-2\ell)!}
    \le
    \sum_{k=0}^{\smash{\lceil n/2 \rceil}} \frac{\xi^{2k}}{k!}
    .
  \end{equation*}
  We treat the two cases of even and odd $n$ separately:
  If $n$ is even, say $n=2m$ for some $m\in \NN_0$, then
  \begin{align*}
    \sum_{\ell=0}^{m} \frac{ \sqrt{(2m)!}\,\xi^{2m-2\ell} }{ 2^{2\ell}\, \ell! \,(2m-2\ell)!}
    &=
    \sum_{k=0}^{m} \frac{ \sqrt{(2m)!}\,\xi^{2k} }{ 2^{2(m-k)}\, (m-k)! \,(2k)!}
    \\
    &=
    \sum_{k=0}^{m} \binom{2m-2k}{m-k}^{\!\frac12} \binom{2m}{2k}^{\!\frac12} \frac{\xi^{2k}}{2^{2(m-k)}\sqrt{(2k)!}}
    \\
    &\le
    \sum_{k=0}^{m} 2^{m-k}\, 2^{m-k}\,\binom{2k}{k}^{\!\frac12} \frac{\xi^{2k}}{2^{2(m-k)}\sqrt{(2k)!}}
    \\
    &=
    \sum_{k=0}^{m} \frac{\xi^{2k}}{k!}
  \end{align*}
  where we substitute $k\coloneqq m-\ell$ and use the standard estimate $\binom{2m-2k}{m-k} \le 2^{2m-2k}$
  and the estimate $\binom{2m}{2k} \le 2^{2m-2k} \binom{2k}{k}$ from Proposition~\ref{proposition:maximalBinomials}.
  If $n$ is odd, i.e., $n=2m +1$ for some $m\in \NN_0$, then
  \begin{align*}
    \sum_{\ell=0}^{m} \frac{ \sqrt{(2m+1)!}\,\abs{\xi}^{2(m-\ell)+1} }{ 2^{2\ell}\, \ell! \,(2(m-\ell)+1)!}
    &=
    \sum_{k=0}^{m} \frac{ \sqrt{(2m+1)!}\,\abs{\xi}^{2k+1} }{ 2^{2(m-k)}\, (m-k)! \,(2k+1)!}
    \\
    &=
    \sum_{k=0}^{m} \binom{2m-2k}{m-k}^{\!\frac12} \binom{2m+1}{2k+1}^{\!\frac12}
    \frac{ \abs{\xi}^{2k+1} }{ 2^{2(m-k)} \sqrt{(2k+1)!} }
    \\
    &\le
    \sum_{k=0}^{m} 2^{m-k} \, 2^{m-k}\,\binom{2k+1}{k}^{\!\frac12}
    \frac{ \abs{\xi}^{2k+1} }{ 2^{2(m-k)} \sqrt{(2k+1)!} }
    \\
    &=
    \sum_{k=0}^{m} \frac{ \abs{\xi}^{2k+1} }{ \sqrt{k!(k+1)!} }
    \\
    &\le
    \frac12\sum_{k=0}^{m} \biggl( \frac{\xi^{2k}}{k!} + \frac{\xi^{2(k+1)}}{(k+1)!} \biggr)
    \\
    &\le
    \sum_{k=0}^{m+1} \frac{\xi^{2k}}{k!}
  \end{align*}
  where we again substitute $k \coloneqq m-\ell$, use the standard estimate $\binom{2m-2k}{m-k} \le 2^{2m-2k}$
  and the estimate $\binom{2m+1}{2k+1} \le 2^{2m-2k} \binom{2k+1}{k}$ from Proposition~\ref{proposition:maximalBinomials},
  and then use $\abs{\xi} \le \frac12\bigl( \sqrt{k+1} + \xi^2 / \sqrt{k+1} \bigr)$, see \eqref{eq:pseudoCS}.
\end{proof}
We can now prove the first estimate \eqref{eq:firstEstimate}:

\begin{proposition} \label{proposition:firstEstimate}
  Let $M\in \NN_0$, $p\in \RR[x_1,\dots,x_d]_{2M}$, and $t\in {[0,1]}$. Then the estimate
  \begin{equation}
    \label{eq:proposition:firstEstimate:estimate}
    \sum_{\substack{\alpha \in \NN_0^d \\ \abs{\alpha} \le 2M}}
    \frac{(\lambda^{\abs{\alpha}}-1)\, H_\alpha\,\skal{H_\alpha}{p}_\infty }{2^{\abs \alpha} \alpha!}
    \preccurlyeq
    \sum_{n = 0}^{M+\lfloor d/2 \rfloor} \frac{(x\cdot x)^n}{n!}
    \preccurlyeq
    \sum_{n = 0}^{M+\lfloor d/2 \rfloor} \frac{(x\cdot x)^n}{(n!)^t}
  \end{equation}
  holds for all $\lambda \in {[0,1]}$ that fulfill
  \begin{equation}
    \label{eq:proposition:firstEstimate:lambda}
    \bigl(1-\lambda^{2M}\bigr)
    \binom{2M+d-1}{d-1}^{\!\frac12} 2^{M+\frac12} \norm{p}
    \le
    1
    .
  \end{equation}
\end{proposition}
\begin{proof}
  Clearly $1/n! \le 1/(n!)^t$ for all $n \in \NN_0$ and $t\in {[0,1]}$, so the second inequality in
  \eqref{eq:proposition:firstEstimate:estimate} certainly holds and we only have to address the first one.
  Let $\lambda \in {[0,1]}$.
  From Proposition~\ref{proposition:HermiteEstimates} and the multinomial theorem \eqref{eq:multinomial} it follows that
  \begin{align*}
    \sum_{\substack{\alpha \in \NN_0^d \\ \abs{\alpha}\le 2M}} \frac{(\lambda^{\abs{\alpha}}-1)\, H_\alpha\, \skal{H_\alpha}{p}_\infty }{2^{\abs \alpha} \alpha!}
    &\preccurlyeq
    \sum_{\substack{\alpha \in \NN_0^d \\ \abs{\alpha}\le 2M}} \abs[\bigg]{ \frac{(\lambda^{\abs{\alpha}}-1) \,\skal{H_\alpha}{p}_\infty }{ \sqrt{\alpha!} } }
    \sum_{\substack{\beta \in \NN_0^d \\ \beta\le \lceil \alpha/2\rceil}} \frac{x^{2\beta}}{\beta!}
    \\
    &\preccurlyeq
    \bigl(1-\lambda^{2M}\bigr)
    \sum_{\substack{\alpha \in \NN_0^d \\ \abs{\alpha}\le 2M}} \frac{\abs[\big]{\skal{H_\alpha}{p}_\infty} }{ \sqrt{\alpha!}}
    \sum_{\substack{\beta \in \NN_0^d \\ \abs\beta\le M+\lfloor d/2 \rfloor}} \frac{x^{2\beta}}{\beta!}
    \\
    &=
    \bigl(1-\lambda^{2M}\bigr)
    \sum_{\substack{\alpha \in \NN_0^d \\ \abs{\alpha}\le 2M}} \frac{\abs[\big]{\skal{H_\alpha}{p}_\infty} }{ \sqrt{\alpha!}}
    \sum_{n=0}^{M+\lfloor d/2 \rfloor} \frac{(x\cdot x)^n}{n!}
    ,
  \end{align*}
  where we use that $\abs{\lceil \alpha / 2 \rceil} \le \abs\alpha/2 + d/2$. By the Cauchy--Schwartz inequality,
  \begin{equation*}
    \sum_{\substack{\alpha \in \NN_0^d \\ \abs{\alpha}\le 2M}} \frac{\abs[\big]{\skal{H_\alpha}{p}_\infty} }{ \sqrt{\alpha!}}
    =
    \sum_{\substack{\alpha \in \NN_0^d \\ \abs{\alpha}\le 2M}} 2^{\abs\alpha/2}\,\frac{\abs[\big]{\skal{H_\alpha}{p}_\infty} }{ 2^{\abs\alpha/2}\sqrt{\alpha!}}
    \le
    \biggl(\sum\nolimits_{\substack{\alpha \in \NN_0^d \\ \abs{\alpha}\le 2M}} 2^{\abs\alpha} \biggr)^{\!\frac12}
    \biggl(\sum\nolimits_{\substack{\alpha \in \NN_0^d \\ \abs{\alpha}\le 2M}} \frac{\skal{H_\alpha}{p}_\infty^2 }{ 2^{\abs\alpha}\alpha!} \biggr)^{\!\frac12}
    .
  \end{equation*}
  By \eqref{eq:orthobasis}, the second factor equals $\skal{p}{p}_\infty^{1/2}$,
  and $\skal{p}{p}_\infty^{1/2} \le \norm{p}$ by Proposition~\ref{proposition:innerproductAndNorm}.
  For the first factor, note that
  $\smash{\sum_{\alpha \in \NN_0^d, \abs{\alpha}\le 2M} 2^{\abs{\alpha}} \le \binom{2M+d-1}{d-1} 2^{2M+1}}$,
  because, by the stars and bars formula, $\smash{N_k \coloneqq \binom{k+d-1}{d-1}}$ is the number of multi-indices $\alpha \in\NN_0^d$
  with $\abs{\alpha} = k\in\NN_0$, and clearly $N_k \le N_{k'}$ for $k,k' \in \NN_0$ with $k\le k'$, so that
  \begin{equation*}
    \sum_{\substack{\alpha \in \NN_0^d \\ \abs{\alpha}\le 2M}} 2^{\abs{\alpha}}
    =
    \sum_{k=0}^{\smash{2M}} N_k\, 2^k \le N_{2M} \sum_{k=0}^{\smash{2M}} 2^k \le N_{2M} \,2^{2M+1} = \binom{2M+d-1}{d-1} \,2^{2M+1}
    .
  \end{equation*}
  Combining these estimates yields
  \begin{equation*}
    \sum_{\substack{\alpha \in \NN_0^d \\ \abs{\alpha}\le 2M}}
    \!\!\!\!\frac{(\lambda^{\abs{\alpha}}-1)\,H_\alpha\, \skal{H_\alpha}{p}_\infty }{2^{\abs \alpha} \alpha!}
    \preccurlyeq
    \bigl(1-\lambda^{2M}\bigr)
    \binom{2M+d-1}{d-1}^{\!\frac12} 2^{M+\frac12} \norm{p}
    \!\sum_{n=0}^{M+\lfloor d/2 \rfloor} \!\frac{(x\cdot x)^n}{n!}
    \preccurlyeq
    \!\sum_{n=0}^{M+\lfloor d/2 \rfloor} \!\frac{(x\cdot x)^n}{n!}
  \end{equation*}
  provided that \eqref{eq:proposition:firstEstimate:lambda} is fulfilled.
\end{proof}
\begin{corollary} \label{corollary:firstEstimate}
  Let $M\in \NN_0$, $p\in \RR[x_1,\dots,x_d]_{2M}$, and $t\in {[0,1]}$.
  Like in Theorem~\ref{theorem:effective}, set
  \begin{equation}
    \label{eq:firstEstimate:mu}
    \mu \coloneqq M\,\binom{2M+d-1}{d-1}^{\!\frac{1}{2}} 2^{M+\frac12}.
  \end{equation}
  If $\mu\,\norm{p} > 1$, then
  \begin{equation}
    \label{eq:corollary:firstEstimate:estimate}
    \sum_{\substack{\alpha \in \NN_0^d \\ \abs{\alpha} \le 2M}} \frac{(\lambda^{\abs{\alpha}}-1)\, H_\alpha\, \skal{H_\alpha}{p}_\infty }{2^{\abs \alpha} \alpha!}
    \preccurlyeq
    \sum_{n = 0}^{M+\lfloor d/2 \rfloor} \frac{(x\cdot x)^n}{(n!)^t}
    \quad\quad\text{for}\quad\quad
    \lambda
    \coloneqq
    \biggl(1 - \frac{1}{\mu\,\norm{p}}\biggr)^{\frac{1}{2}}
    .
  \end{equation}
\end{corollary}
\begin{proof}
  By the previous Proposition~\ref{proposition:firstEstimate} we only have to show that
  $(1-\lambda^{2M})(\mu/M) \norm{p} \le 1$,
  or equivalently that $1 - M / (\mu\,\norm{p}) \le \lambda^{2M}$, where
  $\lambda^{2M} = \smash{\bigl(1 - 1 / (\mu\,\norm{p}) \bigr)^M}$.
  This is trivially true if $M=0$.
  Otherwise we consider the smooth function ${]{-\infty},1[} \ni \eta \mapsto (1 - \eta )^M \in {]0,\infty[}$.
  Its first derivative is $-M(1-\eta)^{M-1}$ and its second derivative is $M(M-1)(1-\eta)^{M-2} \ge 0$.
  Therefore this function is convex and its tangent at $\eta = 0$ is $\RR \ni \eta \mapsto 1-M\eta \in \RR$,
  so $1-M\eta \le (1 - \eta )^M$ for all $\eta \in {]{-\infty},1[}$. The special case
  $\eta \coloneqq 1/(\mu\,\norm{p}) \in {]0,1[}$ yields $1 - M / (\mu\,\norm{p}) \le \smash{\bigl(1 - 1 / (\mu\,\norm{p}) \bigr)^M}$
  as required.
\end{proof}
This particular choice of $\lambda$ in Corollary~\ref{corollary:firstEstimate} has the pleasant property that
$1-\lambda^2 = 1/(\mu\,\norm{p})$.

\subsection{The second estimate} \label{sec:secondEstimate}
For simplicity we write
\begin{equation}
  \label{eq:TailDef}
  \Tail_{N,\lambda}(p)
  \coloneqq
  \sum_{\substack{\alpha \in \NN_0^d \\ 2N < \abs{\alpha} \le 4N}}
  \skal[\big]{ \coefficient{\Kernel_{N,\lambda}}{\alpha} }{ p }_\infty \, x^\alpha
\end{equation}
for $N \in \NN_0$, $\lambda \in {[0,1[}$, and $p\in \RR[x_1,\dots,x_d]$.
In this section we will provide an estimate for $\Tail_{N,\lambda}(p)$,
which appears as the tail of higher-order terms in Corollary~\ref{corollary:rad2infty}.

We will also use the shorthand
\begin{equation}
  \label{eq:CIntegral}
  C_\lambda(\alpha,\beta)
  \coloneqq
  \bigg(
    \frac{
      2^{\abs\alpha} \,\lfloor\frac{\alpha}{2}\rfloor! \, \lceil\frac{\alpha}{2}\rceil!
    }{
      (1-\lambda^2)^{\abs\alpha} \,(\alpha!)^2
    }
  \biggr)^{1/2}
  \int_{\RR^d}
  \xi^{\alpha+\beta}
  \exp\biggl( - \frac{(\xi \cdot \xi)}{1-\lambda^2} \biggr)
  \frac{\D^d\xi}{\sqrt{\pi(1-\lambda^2)}^d}
\end{equation}
for $\lambda \in {[0,1[}$ and $\alpha,\beta \in \NN_0^d$.
As $C_\lambda(\alpha,\beta)$ is a product of Gaussian integrals, its exact value can be computed using \eqref{eq:GaussianIntegral}.
In particular $C_\lambda(\alpha,\beta) \ge 0$ for all $\lambda \in {[0,1[}$ and $\alpha,\beta \in \NN_0^d$.

\begin{proposition} \label{proposition:TailExpansion}
  Let $M,N\in \NN_0$, $p\in \RR[x_1,\dots,x_d]_{2M}$, and $\lambda \in {[0,1[}$. Then
  \begin{equation*}
    \begin{split}
      \Tail_{N,\lambda}&(p)
      =
      \\
      &=
      \sum_{n=N+1}^{2N}
      \!2^n\!
      \sum_{k=0}^{2n-2N-1}
      \sum_{\substack{\alpha\in \NN_0^d \\ \abs{\alpha} = k}}
      \sum_{\substack{\beta\in \NN_0^d \\ \abs{\beta} \le 2M}}
      \frac{\bigl(-(x\cdot x)\bigr)^{n-k}}{(n-k)!}
      \biggl(\frac{\lambda^2}{2(1-\lambda^2)}\biggr)^{n-k/2}
      \frac{x^\alpha}{\bigl(\lfloor\frac{\alpha}{2}\rfloor!\,\lceil\frac{\alpha}{2}\rceil!\bigr)^{1/2}}\,
      \coefficient{p}{\beta}\,C_\lambda(\alpha,\beta)
      .
    \end{split}
  \end{equation*}
  In particular if $\lambda = 0$, then $\Tail_{N,\lambda}(p) = 0$.
\end{proposition}
\begin{proof}
  We apply the definitions of $\Tail_{N,\lambda}(p)$,
  $\skal{\argument}{\argument}_\infty$, $\Kernel_{N,\lambda}$, $\exp_N$, and $C_\lambda$,
  see \eqref{eq:TailDef}, \eqref{eq:skalInftyDef}, \eqref{eq:KNDef}, \eqref{eq:expNDef}, and \eqref{eq:CIntegral},
  respectively, and in between we perform some simplifications:
  \begin{align*}
    &\!\!\Tail_{N,\lambda}(p)
    =
    \\
    &=
    \sum_{\substack{\alpha \in \NN_0^d \\ 2N < \abs{\alpha} \le 4N}}
    \skal[\big]{ \coefficient{\Kernel_{N,\lambda}}{\alpha} }{ p }_\infty \, x^\alpha
    \\
    &=
    \sum_{\substack{\alpha \in \NN_0^d \\ 2N < \abs{\alpha} \le 4N}}
    \int_{\RR^d}
    \coefficient[\big]{\Kernel_{N,\lambda}(\xi)}{\alpha}\, x^\alpha \,p(\xi) \exp\bigl( - (\xi\cdot \xi) \bigr)\,
    \frac{\D^d\xi}{\sqrt{\pi}^d}
    \\
    &=
    \sum_{\substack{\alpha \in \NN_0^d \\ 2N < \abs{\alpha} \le 4N}}
    \int_{\RR^d}
    \coefficient[\bigg]{\exp_N\biggl(-\frac{(x\cdot x) \lambda^2  - 2 (x \cdot \xi) \lambda}{1-\lambda^2}\biggr)}{\alpha}\,x^\alpha \, p(\xi) \exp\biggl( - \frac{(\xi \cdot \xi) }{1-\lambda^2} \biggr)\,
    \frac{\D^d\xi}{\sqrt{\pi(1-\lambda^2)}^d}
    \\
    &=
    \sum_{\substack{\alpha \in \NN_0^d \\ 2N < \abs{\alpha} \le 4N}} \sum_{n=0}^{2N}
    \int_{\RR^d}
    \frac{\coefficient[\big]{\bigl(-(x\cdot x) \lambda^2  + 2 (x \cdot \xi) \lambda \bigr)^n }{\alpha}\,x^\alpha}{(1-\lambda^2)^n \, n!}\,
    p(\xi) \exp\biggl( - \frac{(\xi \cdot \xi) }{1-\lambda^2} \biggr)\,
    \frac{\D^d\xi}{\sqrt{\pi(1-\lambda^2)}^d}
    \\
    &=
    \sum_{\substack{\alpha \in \NN_0^d \\ 2N < \abs{\alpha} \le 4N}} \sum_{n=0}^{2N} \,\sum_{k=0}^{n}
    \int_{\RR^d}
    \frac{\coefficient[\big]{\bigl(-(x\cdot x) \lambda^2\bigr)^{n-k}  \bigl(2 (x \cdot \xi) \lambda\bigr)^k}{\alpha}\,x^\alpha}{(1-\lambda^2)^n \, (n-k)!\, k!}\,
    p(\xi) \exp\biggl( - \frac{(\xi \cdot \xi) }{1-\lambda^2} \biggr)\,
    \frac{\D^d\xi}{\sqrt{\pi(1-\lambda^2)}^d}
    \\
    &\stackrel{\!\!\textup{(a)}\!\!}{=}
    \sum_{n=N+1}^{2N}
    \sum_{k=0}^{2n-2N-1}
    \int_{\RR^d}
    \frac{\bigl(-(x\cdot x) \lambda^2\bigr)^{n-k}  \bigl(2 (x \cdot \xi) \lambda\bigr)^k}{(1-\lambda^2)^n \, (n-k)!\, k!}\,
    p(\xi) \exp\biggl( - \frac{(\xi \cdot \xi) }{1-\lambda^2} \biggr)\,
    \frac{\D^d\xi}{\sqrt{\pi(1-\lambda^2)}^d}
    \\
    &\stackrel{\!\!\textup{(b)}\!\!}{=}
    \sum_{n=N+1}^{2N}
    \sum_{k=0}^{2n-2N-1}
    \sum_{\substack{\alpha\in \NN_0^d \\ \abs{\alpha} = k}}
    \sum_{\substack{\beta\in \NN_0^d \\ \abs{\beta} \le 2M}}
    \int_{\RR^d}
    \frac{\bigl(-(x\cdot x) \lambda^2\bigr)^{n-k} \,2^k\,\lambda^k\, x^\alpha\, \xi^\alpha}{(1-\lambda^2)^n \, (n-k)!\, \alpha!}\,
    \coefficient{p}{\beta} \,\xi^\beta\, \exp\biggl( - \frac{(\xi \cdot \xi) }{1-\lambda^2} \biggr)\,
    \frac{\D^d\xi}{\sqrt{\pi(1-\lambda^2)}^d}
    \\
    &=
    \sum_{n=N+1}^{2N}
    \sum_{k=0}^{2n-2N-1}
    \sum_{\substack{\alpha\in \NN_0^d \\ \abs{\alpha} = k}}
    \sum_{\substack{\beta\in \NN_0^d \\ \abs{\beta} \le 2M}}
    \frac{
      \bigl(-(x\cdot x) \lambda^2\bigr)^{n-k} \,2^{k/2}\,\lambda^k\, x^\alpha
    }{
      (1-\lambda^2)^{n-k/2} \, (n-k)!\, \bigl(\lfloor\frac{\alpha}{2}\rfloor!\,\lceil\frac{\alpha}{2}\rceil!\bigr)^{1/2}
    }\,
    \coefficient{p}{\beta}\,C_\lambda(\alpha,\beta)
    \\
    &=
    \sum_{n=N+1}^{2N}
    2^n
    \sum_{k=0}^{2n-2N-1}
    \sum_{\substack{\alpha\in \NN_0^d \\ \abs{\alpha} = k}}
    \sum_{\substack{\beta\in \NN_0^d \\ \abs{\beta} \le 2M}}
    \frac{\bigl(-(x\cdot x)\bigr)^{n-k}}{(n-k)!}
    \biggl(\frac{\lambda^2}{2(1-\lambda^2)}\biggr)^{n-k/2}
    \frac{x^\alpha}{\bigl(\lfloor\frac{\alpha}{2}\rfloor!\,\lceil\frac{\alpha}{2}\rceil!\bigr)^{1/2}}\,
    \coefficient{p}{\beta}\,C_\lambda(\alpha,\beta)
    .
  \end{align*}
  At (a) we use that $\bigl(-(x\cdot x) \lambda^2\bigr)^{n-k} \bigl(2 (x \cdot \xi) \lambda\bigr)^k$
  is homogeneous in $x_1,\dots,x_d$ of degree $2n-k$, and that $2N < 2n-k$ if and only if $n \ge N+1$ and $k \le 2n-2N-1$.
  At (b) we apply the multinomial theorem \eqref{eq:multinomial} to $(x\cdot \xi)^k / k!$,
  and we expand $p(\xi) = \sum_{\beta\in \NN_0^d} \coefficient{p}{\beta}\,\xi^\beta$.
  This essentially completes the proof. In the special case $\lambda = 0$ it follows that $\Tail_{N,\lambda}(p)=0$ because
  the exponent $n-k/2$ of $\lambda^2 / (2(1-\lambda^2))$ is strictly positive for all $n \in \{N+1,\dots,2N\}$
  and $k \in \{0, \dots, 2n-2N-1\}$.
\end{proof}
In order to obtain an upper bound for $\Tail_{N,\lambda}(p)$ with respect to $\preccurlyeq$
we need upper and lower bounds for the monomials $x^\alpha$, especially in those cases where $\alpha \notin (2\NN_0)^d$.

\begin{lemma} \label{lemma:makeXsquareAgain}
  For all $\eta \in {]0,\infty[}$ and $\alpha \in \NN_0^d$ the following estimate holds:
  \begin{equation}
    \label{eq:makeXsquareAgain}
    -
    b(\eta,\alpha)
    \preccurlyeq
    \frac{x^\alpha}{\bigl(\lfloor\frac{\alpha}{2}\rfloor!\,\lceil\frac{\alpha}{2}\rceil!\bigr)^{1/2}}
    \preccurlyeq
    b(\eta,\alpha)
    \quad
    \text{with}
    \quad
    b(\eta,\alpha)
    \coloneqq
    \frac{1}{2^d}
    \sum_{\nu\in\{0,1\}^d}
    \eta^{\abs{\lfloor \frac{\alpha+\nu}2 \rfloor} -\frac{\abs{\alpha}}2}\,
    \frac{
      x^{2\lfloor \frac{\alpha+\nu}2 \rfloor}
    }{
      \lfloor \frac{\alpha+\nu}2 \rfloor!
    }
    .
  \end{equation}
\end{lemma}
\begin{proof}
  Let $\eta \in {]0,\infty[}$.
  By \eqref{eq:productOfBounds} it is sufficient to treat the univariate case only, i.e.,
  \begin{equation}
   \label{eq:makeXsquareAgain:internal}
   \firstInternalTag
    -
    c(\eta,k)
    \preccurlyeq
    \frac{z^k}{\bigl(\lfloor\frac{k}{2}\rfloor!\,\lceil\frac{k}{2}\rceil!\bigr)^{1/2}}
    \preccurlyeq
    c(\eta,k)
    \quad
    \text{with}
    \quad
    c(\eta,k)
    \coloneqq
    \frac{1}{2}
    \sum_{e\in\{0,1\}}
    \eta^{\lfloor \frac{k+e}2 \rfloor - \frac k2} \,
    \frac{
      z^{2\lfloor \frac{k+e}2 \rfloor}
    }{
      \lfloor \frac{k+e}2 \rfloor!
    }
    \quad\text{for $k\in \NN_0$,}
  \end{equation}
  where $\preccurlyeq$ is the partial order on $\RR[z]$ induced by the convex cone $\SOS{\RR[z]}$.
  So let $k\in \NN_0$. We treat the two cases of even and odd $k$ separately.
  If $k$ is even, then $k/2 = \lfloor k/2 \rfloor = \lceil k/2 \rceil = \lfloor (k+1)/2 \rfloor$, so
  $z^k / \bigl(\lfloor\frac{k}{2}\rfloor!\,\lceil\frac{k}{2}\rceil!\bigr)^{1/2} = z^k / (k/2)! = c(\eta,k)$
  and \eqref{eq:makeXsquareAgain:internal} holds.
  If $k$ is odd, say, $k = 2\ell + 1$ for some $\ell \in \NN_0$,
  then $\lfloor k/2 \rfloor = \ell$, and $\lceil k/2 \rceil = \lfloor (k+1)/2 \rfloor = \ell+1$, so
  \begin{equation*}
    c(\eta,k)
    =
    \frac{1}{2}
    \biggl(
      \eta^{-\frac12}\,
      \frac{
        z^{2\ell}
      }{
        \ell!
      }
      +
      \eta^{\frac12}\,
      \frac{
        z^{2(\ell+1)}
      }{
        (\ell+1)!
      }
    \biggr)
    .
  \end{equation*}
  By \eqref{eq:pseudoCS} the estimate
  \begin{equation*}
    -\frac{1}{2}\biggl(
      \Bigl( \frac{\ell+1}{\eta} \Bigr)^{\frac12} + \Bigl( \frac{\eta}{\ell+1} \Bigr)^{\frac12} z^2
    \biggr)
    \preccurlyeq
    z
    \preccurlyeq
    \frac{1}{2}\biggl(
      \Bigl( \frac{\ell+1}{\eta} \Bigr)^{\frac12} + \Bigl( \frac{\eta}{\ell+1} \Bigr)^{\frac12} z^2
    \biggr)
  \end{equation*}
  holds, from which  \eqref{eq:makeXsquareAgain:internal} follows by multiplication with $z^{2\ell} / \bigl( \ell! (\ell+1)! \bigr)^{1/2}$.
\end{proof}

\begin{proposition} \label{proposition:makeXsquareAgain}
  Let $M,N\in \NN_0$, $p\in \RR[x_1,\dots,x_d]_{2M}$, and $\lambda \in {[0,1[}$. Then
  \begin{equation*}
    \begin{split}
      &\!\Tail_{N,\lambda}(p)
      \preccurlyeq
      \\
      &\preccurlyeq
      \sum_{\nu\in\{0,1\}^d}
      \sum_{n=N+1}^{2N}
      \! 2^{n-d}
      \sum_{k=0}^{2n-2N-1}\!\!
      \sum_{\substack{\alpha\in \NN_0^d \\ \abs{\alpha} = k}}
      \sum_{\substack{\beta\in \NN_0^d \\ \abs{\beta} \le 2M}}\!
      \frac{(x\cdot x)^{n-k}}{(n-k)!}
      \biggl(\frac{1}{2(1-\lambda^2)}\biggr)^{n-k + \abs{\lfloor\frac{\alpha+\nu}2\rfloor}}\,
      \frac{x^{2\lfloor\frac{\alpha+\nu}2\rfloor}}{\lfloor\frac{\alpha+\nu}2\rfloor!}\,
      \abs{\coefficient{p}{\beta}}\,C_\lambda(\alpha,\beta)
      .
    \end{split}
  \end{equation*}
\end{proposition}
\begin{proof}
  Apply the estimate of Lemma~\ref{lemma:makeXsquareAgain} with $\eta \coloneqq \lambda^2 / (2(1-\lambda^2))$
  to the formula for $\Tail_{N,\lambda}(p)$ from Proposition~\ref{proposition:TailExpansion}
  and then use that $\lambda^2 / \bigl(2(1-\lambda^2)\bigr) \le 1 / \bigl(2(1-\lambda^2)\bigr)$.
\end{proof}
In the next step we apply an estimate for the integral $C_\lambda(\alpha,\beta)$.

\begin{lemma} \label{lemma:CIntegralEstimate}
  Let $\lambda \in {[0,1[}$ and $\alpha,\beta \in \NN_0^d$. If $\alpha+\beta \notin (2\NN_0)^d$, then $C_\lambda(\alpha+\beta) = 0$.
  In particular, if $\abs\alpha+\abs\beta \notin 2\NN_0$, then $C_\lambda(\alpha+\beta) = 0$.
  Moreover, the estimate $C_\lambda(\alpha,\beta) \le (\beta!)^{1/2}$ holds.
\end{lemma}
\begin{proof}
  We can expand $C_\lambda(\alpha,\beta)$ as a product of Gaussian integrals:
  \begin{equation*}
    C_\lambda(\alpha,\beta)
    =
    \prod_{j=1}^d
    \bigg(
      \frac{
        2^{\alpha_j} \,\lfloor\frac{\alpha_j}{2}\rfloor! \, \lceil\frac{\alpha_j}{2}\rceil!
      }{
        (1-\lambda^2)^{\alpha_j} \,(\alpha_j!)^2
      }
    \biggr)^{1/2}
    \int_{\RR}
    \xi_j^{\alpha_j+\beta_j}
    \exp\biggl( - \frac{\xi_j^2}{1-\lambda^2} \biggr)
    \frac{\D\xi_j}{\sqrt{\pi(1-\lambda^2)}}
  \end{equation*}
  Formula \eqref{eq:GaussianIntegral} for Gaussian integrals shows that $C_\lambda(\alpha+\beta) = 0$
  if $\alpha+\beta \notin (2\NN_0)^d$, and in particular if $\abs\alpha+\abs\beta \notin 2\NN_0$.
  Moreover, estimate \eqref{eq:GaussianIntegralEstimate} yields
  \begin{align*}
    C_\lambda(\alpha,\beta)
    &\le
    \bigg(
      \frac{
        2^{\abs\alpha} \,\lfloor\frac{\alpha}{2}\rfloor! \, \lceil\frac{\alpha}{2}\rceil!
      }{
        (1-\lambda^2)^{\abs\alpha} \,(\alpha!)^2
      }
      \,
      \frac{
        (1-\lambda^2)^{\abs{\alpha+\beta}} \,(\alpha+\beta)!
      }{
        2^{\abs{\alpha+\beta}}
      }
    \biggr)^{\!\frac{1}{2}}
    \\
    &=
    \Biggl(
      \frac{
        (1-\lambda^2)^{\abs{\beta}}\,\beta!
      }{
        2^{\abs{\beta}}
      }
      \binom{\alpha+\beta}{\alpha}
      \,
      \binom{\alpha}{ \lfloor {\alpha}/{2} \rfloor }^{-1}
    \Biggr)^{\!\frac{1}{2}}
    \\
    &\le
    \bigl( (1-\lambda^2)^{\abs{\beta}}\,\beta! \bigr)^{\frac{1}{2}}
    \\
    &\le
    (\beta!)^{\frac{1}{2}}
  \end{align*}
  where we use that $\binom{\alpha+\beta}{\alpha} \le 2^{\abs\beta} \binom{\alpha}{\lfloor \alpha / 2 \rfloor}$
  by Proposition~\ref{proposition:maximalBinomials}.
\end{proof}

\begin{proposition} \label{proposition:CIntegralEstimate}
  Let $M,N\in \NN_0$, $p\in \RR[x_1,\dots,x_d]_{2M}$, and $\lambda \in {[0,1[}$. Then
  \begin{equation*}
    \begin{split}
      &\!\Tail_{N,\lambda}(p)
      \preccurlyeq
      \\
      &\preccurlyeq
      \sum_{\nu\in\{0,1\}^d}
      \sum_{\substack{\beta\in \NN_0^d \\ \abs{\beta} \le 2M}}
      \abs{\coefficient{p}{\beta}}\,(\beta!)^{\frac{1}{2}}\!\!
      \sum_{n=N+1}^{2N}
      \!2^{n-d}\!\!
      \sum_{k\in \mathcal K_\beta(n,N)}\,
      \sum_{\alpha \in \mathcal A_\beta(k)}\!
      \frac{(x\cdot x)^{n-k}}{(n-k)!}
      \biggl(\frac{1}{2(1-\lambda^2)}\biggr)^{n-k + \abs{\lfloor\frac{\alpha+\nu}2\rfloor}}\,
      \frac{x^{2\lfloor\frac{\alpha+\nu}2\rfloor}}{\lfloor\frac{\alpha+\nu}2\rfloor!}\,
    \end{split}
  \end{equation*}
  where
  \begin{align}
    \mathcal K_\beta(n,N) &\coloneqq \set[\big]{k\in \NN_0}{0\le k \le 2n-2N-1 \text{ and }k+\abs\beta \in 2\NN_0}
  \shortintertext{and}
    \mathcal A_\beta(k) &\coloneqq \set[\big]{\alpha\in \NN_0^d}{\abs\alpha = k\text{ and }\alpha+\beta \in (2\NN_0)^d}
    .
  \end{align}
\end{proposition}
\begin{proof}
  Apply Lemma~\ref{lemma:CIntegralEstimate} to the estimate for $\Tail_{N,\lambda}(p)$
  from Proposition~\ref{proposition:makeXsquareAgain}.
\end{proof}
Next we simplify the summation over $\mathcal K_\beta(n,N)$ and $\mathcal A_\beta(k)$:

\begin{lemma} \label{lemma:halfingSubstitution}
  Let $\lambda \in {[0,1[}$, $\nu \in \{0,1\}^d$, $\beta \in \NN_0^d$, $N\in \NN_0$ and $n\in \{N+1,\dots,2N\}$.
  Let $\mathcal K_\beta(n,N)$ and $\mathcal A_\beta(k)$ be defined like in Proposition~\ref{proposition:CIntegralEstimate}.
  Assume that $N \ge \lfloor d/2 \rfloor$. Then
  \begin{equation}
    \label{eq:halfingSubstitution}
    \begin{split}
      \sum_{k\in \mathcal K_\beta(n,N)}
      \sum_{\alpha \in \mathcal A_\beta(k)}
      \frac{(x\cdot x)^{n-k}}{(n-k)!}
      &\biggl(\frac{1}{2(1-\lambda^2)}\biggr)^{n-k + \abs{\lfloor\frac{\alpha+\nu}2\rfloor}}\,
      \frac{x^{2\lfloor\frac{\alpha+\nu}2\rfloor}}{\lfloor\frac{\alpha+\nu}2\rfloor!}
      \preccurlyeq
      \\
      &\quad\quad\preccurlyeq
      (d+1)
      \sum_{m = N-\lfloor d/2\rfloor}^{n+\lceil d/2\rceil}
      \frac{ (x\cdot x)^m}{m!}
      \biggl(\frac{1}{1-\lambda^2}\biggr)^m
      .
    \end{split}
  \end{equation}
\end{lemma}
\begin{proof}
  First note that for all $b\in \NN_0$ and $e\in \{0,1\}$ the maps
  \begin{align*}
    \set[\big]{a\in \NN_0}{a+b\in 2\NN_0} \ni a
    \mapsto
    \biggl\lfloor \frac{a+e}{2} \biggr\rfloor \in \NN_0
  \shortintertext{and}
    \set[\big]{a\in \NN_0}{a+b\in 2\NN_0} \ni a
    \mapsto
    \biggl\lceil \frac{a+e}{2} \biggr\rceil \in \NN_0
  \end{align*}
  are injective: For every $n \in \NN_0$ there are at most two preimages (one odd, one even) under the maps
  $\NN_0 \ni a \mapsto \lfloor \frac{a+e}{2} \rfloor \in \NN_0$
  and
  $\NN_0 \ni a \mapsto \lceil \frac{a+e}{2} \rceil \in \NN_0$,
  but at most one of them fulfills the condition $a+b\in 2\NN_0$.
  Therefore the map
  \begin{equation*}
    \mathcal A_\beta(k)
    \ni
    \alpha
    \mapsto
    \gamma(\alpha)
    \coloneqq
    \biggl\lfloor\frac{\alpha+\nu}2\biggr\rfloor
    \in
    \NN_0^d
  \end{equation*}
  is injective. From $\abs{\alpha+\nu} = k + \abs\nu$ for all $\alpha \in \mathcal{A}_\beta(k)$
  it follows that $\frac{k + \abs\nu - d}2 \le \abs{\lfloor\frac{\alpha+\nu}2\rfloor} \le \frac{k + \abs\nu}2$,
  and therefore $\lceil\frac{k + \abs\nu - d}2\rceil \le \abs{\lfloor\frac{\alpha+\nu}2\rfloor} \le \lfloor\frac{k + \abs\nu}2\rfloor$,
  so
  \begin{equation*}
    \biggl\lfloor\frac{k}2\biggr\rfloor - \biggl\lfloor\frac{d}2\biggr\rfloor
    \le
    \biggl\lceil\frac{k - d}2\biggr\rceil
    \le
    \biggl\lceil\frac{k + \abs\nu - d}2\biggr\rceil
    \le
    \abs[\bigg]{\biggl\lfloor\frac{\alpha+\nu}2\biggr\rfloor}
    \le
    \biggl\lfloor\frac{k + \abs\nu}2\biggr\rfloor
    \le
    \biggl\lfloor\frac{k + d}2\biggr\rfloor
    \le
    \biggl\lfloor\frac{k}2\biggr\rfloor+\biggl\lceil\frac{d}2\biggr\rceil
    .
  \end{equation*}
  Similarly, the map
  \begin{equation*}
    \mathcal K_\beta(n,N)
    \ni
    k
    \mapsto
    \ell(k)
    \coloneqq
    n- \biggl\lceil\frac{k}2\biggr\rceil
    \in
    \NN_0
  \end{equation*}
  is injective, and
  \begin{align*}
    N
    =
    n- \biggl\lceil \frac{2n-2N-1}2\biggr\rceil
    \le
    n- \biggl\lceil\frac{k}2\biggr\rceil
    \le
    n
  \end{align*}
  holds for all $k\in \mathcal K_\beta(n,N)$.
  We obtain \eqref{eq:halfingSubstitution} by substituting $\gamma$ for $\lfloor\frac{\alpha+\nu}2\rfloor$
  (we may add some positive terms in order to obtain a reasonable description of the new range of
  the summation after the substitution),
  substituting $\ell$ for $n- \lceil\frac{k}2\rceil$, and performing some simplifications
  (in particular we apply the multinomial theorem~\eqref{eq:multinomial}):
  \begin{align*}
    \sum_{k\in \mathcal K_\beta(n,N)}
    &\sum_{\alpha \in \mathcal A_\beta(k)}
    \frac{(x\cdot x)^{n-k}}{(n-k)!}
    \biggl(\frac{1}{2(1-\lambda^2)}\biggr)^{n-k + \abs{\lfloor\frac{\alpha+\nu}2\rfloor}}\,
    \frac{x^{2\lfloor\frac{\alpha+\nu}2\rfloor}}{\lfloor\frac{\alpha+\nu}2\rfloor!}
    \preccurlyeq
    \\
    &\preccurlyeq
    \sum_{\Delta = -\lfloor d/2\rfloor}^{\lceil d/2\rceil}
    \sum_{k\in \mathcal K_\beta(n,N)}
    \sum_{\substack{\gamma \in \NN_0^d \\ \abs\gamma = \lfloor k/2 \rfloor + \Delta}}
    \frac{(x\cdot x)^{n-k}}{(n-k)!}
    \biggl(\frac{1}{2(1-\lambda^2)}\biggr)^{n-k + \lfloor k/2 \rfloor + \Delta}\,
    \frac{x^{2\gamma}}{\gamma!}
    \\
    &=
    \sum_{\Delta = -\lfloor d/2\rfloor}^{\lceil d/2\rceil}
    \sum_{\substack{k\in \mathcal K_\beta(n,N) \\ \lfloor k/2 \rfloor + \Delta \ge 0}}
    \frac{(x\cdot x)^{n-k}}{(n-k)!}
    \biggl(\frac{1}{2(1-\lambda^2)}\biggr)^{n-k + \lfloor k/2 \rfloor + \Delta}\,
    \frac{(x\cdot x)^{\lfloor k/2 \rfloor + \Delta}}{(\lfloor k/2 \rfloor + \Delta)!}
    \\
    &=
    \sum_{\Delta = -\lfloor d/2\rfloor}^{\lceil d/2\rceil}
    \sum_{\substack{k\in \mathcal K_\beta(n,N) \\ \lfloor k/2 \rfloor + \Delta \ge 0}}
    \frac{(x\cdot x)^{n-k+\lfloor k/2 \rfloor + \Delta}}{(n-k+\lfloor k/2 \rfloor + \Delta)!} \binom{n-k+\lfloor k/2 \rfloor + \Delta}{n-k}
    \biggl(\frac{1}{2(1-\lambda^2)}\biggr)^{n-k + \lfloor k/2 \rfloor + \Delta}\,
    \\
    &\stackrel{\text{(a)}}{\preccurlyeq}
    \sum_{\Delta = -\lfloor d/2\rfloor}^{\lceil d/2\rceil}
    \sum_{k\in \mathcal K_\beta(n,N) }
    \frac{(x\cdot x)^{n-\lceil k/2 \rceil + \Delta}}{(n-\lceil k/2 \rceil + \Delta)!} \binom{n-\lceil k/2 \rceil + \Delta}{n-k}
    \biggl(\frac{1}{2(1-\lambda^2)}\biggr)^{n-\lceil k/2 \rceil + \Delta}\,
    \\
    &\preccurlyeq
    \sum_{\Delta = -\lfloor d/2\rfloor}^{\lceil d/2\rceil}
    \sum_{k\in \mathcal K_\beta(n,N) }
    \frac{(x\cdot x)^{n-\lceil k/2 \rceil + \Delta}}{(n-\lceil k/2 \rceil + \Delta)!}
    \biggl(\frac{1}{1-\lambda^2}\biggr)^{n-\lceil k/2 \rceil + \Delta}\,
    \\
    &=
    \sum_{\Delta = -\lfloor d/2\rfloor}^{\lceil d/2\rceil}\,
    \sum_{\ell = N}^n
    \frac{ (x\cdot x)^{\ell + \Delta}}{(\ell + \Delta)!}
    \biggl(\frac{1}{1-\lambda^2}\biggr)^{\ell + \Delta}\,
    \\
    &\preccurlyeq
    (d+1)
    \sum_{m = N-\lfloor d/2\rfloor}^{n+\lceil d/2\rceil}
    \frac{ (x\cdot x)^m}{m!}
    \biggl(\frac{1}{1-\lambda^2}\biggr)^m
  \end{align*}
  where (a) is not an equality because we drop the restriction $\lfloor k/2 \rfloor + \Delta \ge 0$;
  this does not introduce ill-defined terms
  because $n - \lceil k/2 \rceil + \Delta \ge N + \Delta \ge \lfloor d/2 \rfloor + \Delta \ge 0$
  due the assumption that $N \ge \lfloor d/2 \rfloor$.
\end{proof}

\begin{proposition} \label{proposition:halfingSubstitution}
  Let $M,N\in \NN_0$, $p\in \RR[x_1,\dots,x_d]_{2M}$, and $\lambda \in {[0,1[}$.
  Assume that $N \ge \lfloor d / 2 \rfloor$. Then
  \begin{equation}
    \Tail_{N,\lambda}(p)
    \preccurlyeq
    (d+1)\,2^{d+1}\,\norm{p}
    \sum_{m = N-\lfloor d/2\rfloor}^{2N+\lceil d/2\rceil}
    \frac{ (x\cdot x)^m}{m!}
    \biggl(\frac{4}{1-\lambda^2}\biggr)^m
    .
  \end{equation}
\end{proposition}
\begin{proof}
  We apply the previous Lemma~\ref{lemma:halfingSubstitution} to the estimate for
  $\Tail_{N,\lambda}(p)$ from Proposition~\ref{proposition:CIntegralEstimate},
  and then simplify the result:
  \begin{align*}
    \Tail_{N,\lambda}(p)
    &\preccurlyeq
    \sum_{\nu\in\{0,1\}^d}
    \sum_{\substack{\beta\in \NN_0^d \\ \abs{\beta} \le 2M}}
    \abs{\coefficient{p}{\beta}}\,(\beta!)^{\frac{1}{2}}
    \sum_{n=N+1}^{2N}
    2^{n-d}
    (d+1)
    \sum_{m = N-\lfloor d/2\rfloor}^{n+\lceil d/2\rceil}
    \frac{ (x\cdot x)^m}{m!}
    \biggl(\frac{1}{1-\lambda^2}\biggr)^m
    \\
    &=
    (d+1)\,\norm{p}
    \sum_{n=N+1}^{2N}
    2^{n}
    \sum_{m = N-\lfloor d/2\rfloor}^{n+\lceil d/2\rceil}
    \frac{ (x\cdot x)^m}{m!}
    \biggl(\frac{1}{1-\lambda^2}\biggr)^m
    \\
    &\preccurlyeq
    (d+1)\,\norm{p}\,2^{2N+1}
    \sum_{m = N-\lfloor d/2\rfloor}^{2N+\lceil d/2\rceil}
    \frac{ (x\cdot x)^m}{m!}
    \biggl(\frac{1}{1-\lambda^2}\biggr)^m
    \\
    &\preccurlyeq
    (d+1)\,\norm{p}
    \sum_{m = N-\lfloor d/2\rfloor}^{2N+\lceil d/2\rceil}
    2^{2m+d+1}
    \frac{ (x\cdot x)^m }{ m! }
    \biggl(\frac{1}{1-\lambda^2}\biggr)^m
    \\
    &\preccurlyeq
    (d+1)\,2^{d+1}\,\norm{p}
    \sum_{m = N-\lfloor d/2\rfloor}^{2N+\lceil d/2\rceil}
    \frac{ (x\cdot x)^m }{ m! }
    \biggl(\frac{4}{1-\lambda^2}\biggr)^m
    .
  \end{align*}
\end{proof}
For $m\to \infty$, the factor $\frac{1}{m!} \bigl( 4 / (1-\lambda^2)\bigr)^m$ tends to $0$. For a quantitative
result, we apply a variant of Stirling's formula, see \cite{robbins} or \cite{JamesonStirling}:
\begin{equation}
  \label{eq:stirling}
  \Gamma(1+\xi) > \sqrt{2\pi \xi } \,\biggl( \frac{\xi}{\E} \biggr)^\xi
  \quad\quad\text{for all $\xi \in {]0,\infty[}$}
\end{equation}
where $\E$ is Euler's number and $\Gamma \colon {]0,\infty[} \to {]0,\infty[}$ the unique logarithmically convex function
that fulfills the fundamental equation $\Gamma(1+\xi) = \xi \,\Gamma(\xi)$
for all $\xi \in {]0,\infty[}$ and $\Gamma(1) = 1$ (this is the characterization of the Gamma function from the Bohr–Mollerup theorem).
In particular $\Gamma(1+n) = n!$ for all $n\in \NN_0$.

\begin{lemma} \label{lemma:Gamma}
  Let $\xi \in {]0,\infty[}$ and $\eta \in {[0,\infty[}$. Then the identity $\Gamma(1+\xi+\eta) \ge \xi^\eta\, \Gamma(1+\xi)$ holds.
\end{lemma}
\begin{proof}
  Set $\lambda \coloneqq \eta - \lfloor \eta \rfloor \in {[0,1[}$.
  Note that $1+\xi = \lambda(\xi+\lambda) + (1-\lambda)(1+\xi+\lambda)$. By logarithmic convexity of $\Gamma$ it follows that
  \begin{align*}
    \ln\bigl(\Gamma(1+\xi)\bigr)
    &\le
    \lambda \ln\bigl(\Gamma(\xi+\lambda)\bigr)
    +
    (1-\lambda) \ln\bigr(\Gamma(1+\xi+\lambda)\bigr)
    ,
  \shortintertext{or equivalently}
    \Gamma(1+\xi)
    &\le
    \Gamma(\xi+\lambda)^\lambda\,\Gamma(1+\xi+\lambda)^{1-\lambda}
    .
  \end{align*}
  Multiplication with $(\xi+\lambda)^\lambda$ and application of the fundamental equation $(\xi+\lambda)\, \Gamma(\xi+\lambda) = \Gamma(1+\xi+\lambda)$
  then yields $(\xi+\lambda)^\lambda \,\Gamma(1+\xi) \le \Gamma(1+\xi+\lambda)$. Therefore
  \begin{align*}
    \xi^\eta \,\Gamma(1+\xi)
    &\le
    \Bigl(\prod\nolimits_{j=1}^{\lfloor \eta\rfloor} (j+\xi+\lambda)\Bigr)\,(\xi+\lambda)^\lambda\, \Gamma(1+\xi)
    \\
    &\le
    \Bigl(\prod\nolimits_{j=1}^{\lfloor \eta\rfloor} (j+\xi+\lambda)\Bigr)\, \Gamma(1+\xi+\lambda)
    \\
    &=
    \Gamma\bigl(1+\lfloor\eta\rfloor + \xi+\lambda \bigr)
    \\
    &=
    \Gamma(1+\xi+\eta)
  \end{align*}
  where we apply the fundamental equation $(j+\xi+\lambda)\, \Gamma(j+\xi+\lambda) = \Gamma(1+j+\xi+\lambda)$
  with $j \in \{1,\dots,\lfloor \eta \rfloor\}$.
\end{proof}

\begin{lemma} \label{lemma:stirlingGamma}
  Let $a,b \in {]0,\infty[}$ and $m\in \NN_0$. Assume
  \begin{equation}
    m \ge \E b + \max\biggl\{ 0, \ln \biggl( \frac{a}{\sqrt{b}} \biggr) \biggr\}
    ,
  \end{equation}
  where $\E$ is Euler's number. Then $m! \ge ab^m$.
\end{lemma}
\begin{proof}
  Note that $m\ge \E b > 0$ and $m- \E b \ge \ln\bigl( a / \sqrt{b} \bigr) \ge \ln\bigl( a / \sqrt{2\pi \E b} \bigr)$.
  We combine the previous Lemma~\ref{lemma:Gamma} for $\xi \coloneqq \E b > 0$ and $\eta \coloneqq m-\E b \ge 0$
  with Stirling's formula~\eqref{eq:stirling}:
  \begin{equation*}
    m!
    =
    \Gamma(1+m) \ge (\E b)^{m-\E b} \, \Gamma(1+\E b)
    \ge
    (\E b)^{m-\E b}\sqrt{2\pi \E b} \,b^{\E b}
    =
    \E^{m-\E b}\sqrt{2\pi \E b}\, b^m \ge a b^m
    .
  \end{equation*}
\end{proof}
We can now prove the second estimate \eqref{eq:secondEstimate}:

\begin{proposition} \label{proposition:secondEstimate}
  Let $M,N\in \NN_0$, $p\in \RR[x_1,\dots,x_d]_{2M}$, and $\lambda,t \in {[0,1[}$.
  Assume that
  \begin{equation}
    \label{eq:proposition:secondEstimate:M}
    N
    \ge
    \biggl\lfloor \frac{d}{2}\biggr\rfloor
    +
    \E \biggl(\frac{ 4 }{1-\lambda^2} \biggr)^{\frac{1}{1-t}}
    +
    \max\Biggl\{
      0
      ,
      \frac{
        \ln\bigl( (d+1)2^d\norm{p} \sqrt{1-\lambda^2} \bigr)
      }{1-t}
    \Biggr\}
    .
  \end{equation}
  Then
  \begin{equation}
    \Tail_{N,\lambda}(p)
    \preccurlyeq
    \sum_{m = N-\lfloor d/2\rfloor}^{2N+\lceil d/2\rceil}
    \frac{ (x\cdot x)^m}{(m!)^t}
    .
  \end{equation}
\end{proposition}
\begin{proof}
  This is certainly true if $p=0$. Otherwise, set
  \begin{equation*}
    a \coloneqq \bigl( (d+1)\,2^{d+1}\,\norm{p} \bigr)^{\frac{1}{1-t}} \in {]0,\infty[}
    \quad\quad\text{and}\quad\quad
    b \coloneqq \biggl( \frac{4}{1-\lambda^2} \biggr)^{\frac{1}{1-t}} \in {]0,\infty[}
  \end{equation*}
  and note that assumption \eqref{eq:proposition:secondEstimate:M}
  is equivalent to $N - \lfloor d/2 \rfloor \ge \E b + \max \bigl\{0, \ln(a/\sqrt{ b}) \bigr\}$.
  Then by Proposition~\ref{proposition:halfingSubstitution} and by the previous Lemma~\ref{lemma:stirlingGamma}, the estimate
  \begin{align*}
    \Tail_{N,\lambda}(p)
    &\preccurlyeq
    (d+1)\,2^{d+1}\,\norm{p}
    \sum_{m = N-\lfloor d/2\rfloor}^{2N+\lceil d/2\rceil}
    \frac{ (x\cdot x)^m}{m!}
    \biggl( \frac{4}{1-\lambda^2} \biggr)^m
    \\
    &=
    \sum_{m = N-\lfloor d/2\rfloor}^{2N+\lceil d/2\rceil}
    \frac{ (x\cdot x)^m}{(m!)^t}\,
    \biggl( \frac{ab^m}{m!} \biggr)^{1-t}
    \\
    &\preccurlyeq
    \sum_{m = N-\lfloor d/2\rfloor}^{2N+\lceil d/2\rceil}
    \frac{ (x\cdot x)^m}{(m!)^t}\,
  \end{align*}
  holds.
\end{proof}

\subsection{Synthesis}
\label{sec:synthesis}
We can now apply the explicit estimates obtained in Sections~\ref{sec:firstEstimate} and \ref{sec:secondEstimate}.

\begin{proof}[of Theorem~\ref{theorem:effective}]
  Let $M\in \NN$ and $p\in \RR[x_1,\dots,x_d]_{2M}$, and assume that $p(\xi) \ge 0$ for all $\xi \in \RR^d$.
  Let $t\in {[0,1[}$.
  It is sufficient to treat only the special case $\epsilon = 1$;
  the case of general $\epsilon \in {]0,\infty[}$ then follows by applying this special case to the rescaled polynomial $p/\epsilon$.
  By Proposition~\ref{proposition:SOS} and Corollary~\ref{corollary:rad2infty},
  \begin{equation*}
    0
    \preccurlyeq
    \IOp_{N,\lambda,\infty}(p)
    =
    p
    +
    \sum_{\substack{\alpha\in \NN_0^d \\ \abs{\alpha} \le 2M}}
    \frac{(\lambda^{\abs \alpha}-1) \,H_\alpha\,\skal{H_\alpha}{p}_\infty}{2^{\abs \alpha} \alpha!}
    +
    \Tail_{N,\lambda}(p)
    \quad\text{for all $\lambda \in {[0,1[}$ and $N\in \NN_0$, $N \ge M$,}
  \end{equation*}
  with $\Tail_{N,\lambda}(p)$ from \eqref{eq:TailDef}.
  Define $\mu \in {]0,\infty[}$ like in \eqref{eq:effective:mu} or equivalently \eqref{eq:firstEstimate:mu}.
  If $\mu\,\norm{p} \le 1$, then condition \eqref{eq:proposition:firstEstimate:lambda}
  of Proposition~\ref{proposition:firstEstimate} holds for all $\lambda \in {[0,1]}$
  (note that we assume $M\ge 1$ here), so
  \begin{equation*}
    \sum_{\substack{\alpha \in \NN_0^d \\ \abs{\alpha} \le 2M}} \frac{(\lambda^{\abs{\alpha}}-1)\,H_\alpha\, \skal{H_\alpha}{p}_\infty }{2^{\abs \alpha} \alpha!}
    \preccurlyeq
    \sum_{m = 0}^{M+\lfloor d/2 \rfloor} \frac{(x\cdot x)^m}{(m!)^t}
  \end{equation*}
  for all $\lambda \in {[0,1]}$ by Proposition~\ref{proposition:firstEstimate}, and in particular for $\lambda = 0$.
  As $\Tail_{N,0}(p) = 0$ for all $N\in \NN_0$ by Proposition~\ref{proposition:TailExpansion},
  it then follows that
  \begin{equation*}
    0
    \preccurlyeq
    \IOp_{M,0,\infty}(p)
    =
    p
    +
    \sum_{\substack{\alpha\in \NN_0^d \\ \abs{\alpha} \le 2M}}
    \frac{(\lambda^{\abs \alpha}-1)\, H_\alpha\,\skal{H_\alpha}{p}_\infty}{2^{\abs \alpha} \alpha!}
    +
    \Tail_{M,0}(p)
    \preccurlyeq
    p
    +
    \sum_{m = 0}^{M+\lfloor d/2 \rfloor} \frac{(x\cdot x)^m}{(m!)^t}
    .
  \end{equation*}
  As $M+\lfloor d/2 \rfloor \le \Nexpl(p)$, this completes the proof in the case $\mu\,\norm{p} \le 1$.
  Otherwise $\mu\,\norm{p} > 1$. In this case
  \begin{equation*}
    \sum_{\substack{\alpha \in \NN_0^d \\ \abs{\alpha} \le 2M}} \frac{(\lambda^{\abs{\alpha}}-1)\,H_\alpha\, \skal{H_\alpha}{p}_\infty }{2^{\abs \alpha} \alpha!}
    \preccurlyeq
    \sum_{m = 0}^{M+\lfloor d/2 \rfloor} \frac{(x\cdot x)^m}{(m!)^t}
    \quad\quad\text{for}\quad\quad
    \lambda \coloneqq \biggl(1-\frac{1}{\mu\,\norm{p}}\biggr)^{\!\frac12}
  \end{equation*}
  by Corollary~\ref{corollary:firstEstimate}, and
  \begin{equation*}
    \Tail_{N,\lambda}(p)
    \preccurlyeq
    \sum_{m = N-\lfloor d/2\rfloor}^{2N+\lceil d/2\rceil}
    \frac{ (x\cdot x)^m}{(m!)^t}
  \end{equation*}
  by Proposition~\ref{proposition:secondEstimate} for all those $N\in \NN_0$ that fulfill
  \begin{equation*}
    N
    \ge
    \biggl\lfloor \frac{d}{2}\biggr\rfloor
    +
    \E \biggl(\frac{ 4 }{1-\lambda^2} \biggr)^{\frac{1}{1-t}}
    +
    \max\Biggl\{
      0
      ,
      \frac{
        \ln\bigl( (d+1)2^d\norm{p} \sqrt{1-\lambda^2} \bigr)
      }{1-t}
    \Biggr\}
    .
  \end{equation*}
  As $1-\lambda^2 = \bigl(\mu\,\norm{p}\bigr)^{-1}$, this condition on $N$ can equivalently be expressed as
  \begin{equation}
    \secondInternalTag
    \label{eq:effectiveInternal:N}
    N
    \ge
    \biggl\lfloor \frac{d}{2}\biggr\rfloor
    +
    \E \bigl( 4\mu\,\norm{p} \bigr)^{\frac{1}{1-t}}
    +
    \max\Biggl\{
      0
      ,
      \frac{
        \ln\bigl( (d+1)2^d\sqrt { \norm{p} / \mu } \bigr)
      }{1-t}
    \Biggr\}
    .
  \end{equation}
  In total,
  \begin{align*}
    0
    \preccurlyeq
    \IOp_{N,\lambda,\infty}(p)
    =
    p
    +
    \sum_{\substack{\alpha\in \NN_0^d \\ \abs{\alpha} \le 2M}}
    \frac{(\lambda^{\abs \alpha}-1)\, H_\alpha\,\skal{H_\alpha}{p}_\infty}{2^{\abs \alpha} \alpha!}
    +
    \Tail_{N,\lambda}(p)
    \preccurlyeq
    p
    +
    \sum_{m = 0}^{2N+\lceil d/2\rceil}
    \frac{ (x\cdot x)^m}{(m!)^t}
  \end{align*}
  provided that the conditions $M \le N$, $M+\lfloor d/2 \rfloor < N-\lfloor d/2 \rfloor$ and \eqref{eq:effectiveInternal:N} are fulfilled.
  This is the case if we choose $N \in \NN_0$ as the smallest integer that fulfills $N \ge M+d + 1$
  and \eqref{eq:effectiveInternal:N}. Then $2N + \lceil d/2 \rceil \le \Nexpl(p)$ because
  \begin{equation*}
    2\bigl(M + d + 1\bigr) + \biggl\lceil \frac d2 \biggr\rceil
    =
    2M + \biggl\lceil \frac{5d}{2} \biggr\rceil + 2
  \end{equation*}
  and because
  \begin{align*}
    2\Biggl(
    \biggl\lfloor \frac{d}{2}\biggr\rfloor
    &+
    \Biggl\lceil
    \E \bigl( 4\mu\,\norm{p} \bigr)^{\frac{1}{1-t}}
    +
    \max\Biggl\{
      0
      ,
      \frac{
        \ln\bigl( (d+1)2^d\sqrt { \norm{p} / \mu } \bigr)
      }{1-t}
    \Biggr\}
    \Biggr\rceil
    \Biggr)
    +
    \biggl\lceil \frac{d}{2}\biggr\rceil
    =
    \\
    &=
    2\Biggl\lceil
    \E \bigl( 4\mu\,\norm{p} \bigr)^{\frac{1}{1-t}}
    +
    \max\Biggl\{
      0
      ,
      \frac{
        \ln\bigl( (d+1)2^d\sqrt { \norm{p} / \mu } \bigr)
      }{1-t}
    \Biggr\}
    \Biggr\rceil
    +
    d+\biggl\lfloor \frac{d}{2}\biggr\rfloor
    \\
    &\le
    \Bigl\lceil
    2\E \bigl( 4\mu\,\norm{p} \bigr)^{\frac{1}{1-t}}
    \Bigr\rceil
    +
    \max\Biggl\{
      0
      ,
      \biggl\lceil
      \frac{1}{1-t}
      \ln\biggl( \frac{ (d+1)^2\,2^{2d}\, \norm{p} }{ \mu } \biggr)
      \biggr\rceil
    \Biggr\}
    +
    \biggl\lfloor \frac{3d}{2} \biggr\rfloor
    +
    1
    .
  \end{align*}
\end{proof}

\section{Conclusion and perspectives}
\label{sec:conclusion}
Our work focused on approximations of nonnegative polynomials by sums of squares (SOS) using a
perturbative approach inspired by Lasserre’s framework.
We developed a quantitative perturbative  Positivstellensatz on $\RR^d$ with an explicit bound on
the truncation order $N$ 
ensuring the existence of
SOS certificates of the form
\begin{equation*}
  p+\varepsilon\sum_{n=0}^N \frac{(x\cdot x)^n}{(n!)^t}\in \SOS{\RR[x_1,\dots,x_d]}, \qquad 0<t<1,
\end{equation*}
whenever $p\geq0$ on $\RR^d$.
The main contribution, Theorem \ref{theorem:effective}, provides the polynomial growth bound 
\begin{equation*}
N = O\bigl((\|p\|/\varepsilon)^{1/(1-t)}\bigr),
\end{equation*}
obtained via a positive integration kernel built from Hermite polynomials (using Mehler's formula)
in the spirit of the polynomial kernel method.
Our result strengthens the theoretical foundation for approximating nonnegative polynomials in unbounded domains,
where traditional SOS methods often face challenges.

For future work, several directions appear promising.
First, it would be interesting to extend the kernel-based
construction to other positivity settings where perturbative results are known, such as for positive noncommutative polynomials
\cite{navascues2013paradox}, trace-positive noncommutative polynomials \cite{klep2023globally}, and positive moment polynomials
\cite{klep2025sums}.

Another important future work is to obtain explicit convergence rates for the Lasserre hierarchy of \emph{upper bounds} for
unconstrained polynomial optimization with respect to a Gaussian reference measure.
To the best of our knowledge, this would yield the first explicit convergence rate guarantees for SOS-based polynomial optimization
hierarchies on noncompact domains.
Lasserre hierarchy of upper bounds is obtained by minimizing the expectation of the objective polynomial with respect to polynomial
densities that are sums of squares of bounded degree relative to a fixed reference measure supported on the domain
\cite{lasserre2011new}.
While asymptotic convergence of these upper bounds to the global minimum is guaranteed under mild assumptions, a major focus of recent work has been on understanding quantitative convergence rates
\cite{deKlerk2017,deKlerk2020,deKlerk2022}. 
In our context, we would consider the problem of minimizing a polynomial $p \in \RR[x_1,\dots,x_d]$ over $\RR^n$, i.e., computing 
\begin{align}
\label{eq:pmin}
p_{\min} = \inf_{\xi \in \RR^d } p(\xi).
\end{align} 
An alternative way to formulate \eqref{eq:pmin} is the following:
\begin{align}
\label{eq:pminmeas}
p_{\min}
=
\inf_{\Phi \in \mathcal{M}_{+}(\RR^d)}
\set[\bigg]{
  \int_{\RR^d} p(\xi) \D \Phi(\xi)
}{
  \int_{\RR^d} \D \Phi(\xi) = 1
}
,
\end{align}
where $\mathcal{M}_+(\RR^d)$ denotes the set of nonnegative Borel measures supported on $\RR^d$. 
Instead of optimizing over the full set $\mathcal{M}_+(\RR^d)$, the framework developed by Lasserre in \cite{lasserre2011new} consists of optimizing over measures of the form $\D \Phi(\xi) = \sigma(\xi) \D \Psi(\xi)$, where $\Psi$ is a fixed reference measure with support being exactly $\RR^d$, and $\sigma$ is a polynomial that is pointwise nonnegative on $\RR^d$. 
This yields the so-called hierarchy of upper bounds for $p_{\min}$, indexed by $N \in \NN_0$:
\begin{align}
\label{eq:upperhierarchy}
\ub{(p,\Psi)}_N
\coloneqq
\inf_{\sigma \in C_{N}}
\set[\bigg]{
  \int_{\RR^d} p(\xi) \sigma(\xi) \D \Psi(\xi)
}{
  \int_{\RR^d} \sigma(\xi) \D \Psi(\xi) = 1
}
,
\end{align}
where $C_{N}$ is the convex cone  $\SOS{\RR[x_1,\dots,x_d]}\cap \RR[x_1,\dots,x_d]_{2N}$. 
We would like to follow the procedure of \cite[Sec.~5]{slot2022sum} by exhibiting an explicit reference measure $\Psi$ with support equal to $\RR^d$,  and a sequence of probability densities $(\sigma_N)_{N \in \NN_0}$, where each $\sigma_N \in C_N$, and for which the difference
$$
\int_{\RR^d} p(\xi) \sigma_N(x) \D \Psi(\xi) - p_{\min}
$$
can be bounded from above. However,
one cannot directly apply this procedure because our kernel is not symmetric under exchanging the input and output parameters $x$
and $\xi$.
Overcoming this roadblock motivates a dedicated further study. 

\end{onehalfspace}


\begin{thebibliography}{NGSA{\etalchar{+}}13}

\bibitem[AAR00]{askey}
George~E. Andrews, Richard Askey, and Ranjan Roy.
\newblock {\em Special functions}, volume~71 of {\em Encycl. Math. Appl.}
\newblock Cambridge University Press, paperback edition, 2000.

\bibitem[BM23]{baldi2023effective}
Lorenzo Baldi and Bernard Mourrain.
\newblock {On the effective Putinar's Positivstellensatz and moment
  approximation}.
\newblock {\em Math. Program.}, 200(1):71--103, 2023.

\bibitem[BS24]{baldi2024degree}
Lorenzo Baldi and Lucas Slot.
\newblock {Degree bounds for Putinar’s Positivstellensatz on the hypercube}.
\newblock {\em SIAM J. Appl. Algebra Geom.}, 8(1):1--25, 2024.

\bibitem[Dav25]{davidson}
Kenneth~R. Davidson.
\newblock {\em Functional analysis and operator algebras}, volume~13 of {\em
  CMS/CAIMS Books Math.}
\newblock Springer, 2025.

\bibitem[dKL20]{deKlerk2020}
Etienne de~Klerk and Monique Laurent.
\newblock Worst-case examples for {Lasserre}'s measure-based hierarchy for
  polynomial optimization on the hypercube.
\newblock {\em Math. Oper. Res.}, 45(1):86--98, 2020.

\bibitem[dKL22]{deKlerk2022}
Etienne de~Klerk and Monique Laurent.
\newblock Convergence analysis of a {Lasserre} hierarchy of upper bounds for
  polynomial minimization on the sphere.
\newblock {\em Math. Program.}, 193(2 (B)):665--685, 2022.

\bibitem[dKLS17]{deKlerk2017}
Etienne de~Klerk, Monique Laurent, and Zhao Sun.
\newblock Convergence analysis for {Lasserre}'s measure-based hierarchy of
  upper bounds for polynomial optimization.
\newblock {\em Math. Program.}, 162(1-2 (A)):363--392, 2017.

\bibitem[FF21]{fawzi}
Kun Fang and Hamza Fawzi.
\newblock The sum-of-squares hierarchy on the sphere and applications in
  quantum information theory.
\newblock {\em Math. Program.}, 190(1-2 (A)):331--360, 2021.

\bibitem[JAM15]{JamesonStirling}
G.~J.~O. JAMESON.
\newblock A simple proof of stirling's formula for the gamma function.
\newblock {\em Math. Gaz.}, 99(544):68--74, 2015.

\bibitem[KMRZ25]{korda2025convergence}
Milan Korda, Victor Magron, and Rodolfo Rios-Zertuche.
\newblock Convergence rates for sums-of-squares hierarchies with correlative
  sparsity.
\newblock {\em Math. Program.}, 209(1):435--473, 2025.

\bibitem[KMV25]{klep2025sums}
Igor Klep, Victor Magron, and Jurij Vol{\v{c}}i{\v{c}}.
\newblock Sums of squares certificates for polynomial moment inequalities.
\newblock {\em Found. Comput. Math.}, pages 1--43, 2025.

\bibitem[KSV23]{klep2023globally}
Igor Klep, Claus Scheiderer, and Jurij Vol{\v{c}}i{\v{c}}.
\newblock Globally trace-positive noncommutative polynomials and the unbounded
  tracial moment problem.
\newblock {\em Math. Ann.}, 387(3):1403--1433, 2023.

\bibitem[Las01]{lasserre2001global}
Jean~B. Lasserre.
\newblock Global optimization with polynomials and the problem of moments.
\newblock {\em SIAM J. Optim.}, 11(3):796--817, 2001.

\bibitem[Las07]{lasserre}
Jean~B. Lasserre.
\newblock A sum of squares approximation of nonnegative polynomials.
\newblock {\em SIAM Rev.}, 49(4):651--669, 2007.

\bibitem[Las11]{lasserre2011new}
Jean~B. Lasserre.
\newblock A new look at nonnegativity on closed sets and polynomial
  optimization.
\newblock {\em SIAM J. Optim.}, 21(3):864--885, 2011.

\bibitem[Lau09]{laurent}
Monique Laurent.
\newblock Sums of squares, moment matrices and optimization over polynomials.
\newblock In {\em Emerging applications of algebraic geometry. Papers of the
  IMA workshops Optimization and control, January 16--20, 2007 and Applications
  in biology, dynamics, and statistics, March 5--9, 2007, held at IMA,
  Minneapolis, MN, USA}, pages 157--270. New York, NY: Springer, 2009.

\bibitem[LM19]{lasserre2019sdp}
Jean~B. Lasserre and Victor Magron.
\newblock In {SDP} relaxations, inaccurate solvers do robust optimization.
\newblock {\em SIAM J. Optim.}, 29(3):2128--2145, 2019.

\bibitem[LS23]{laurent2023effective}
Monique Laurent and Lucas Slot.
\newblock {An effective version of Schm{\"u}dgen’s Positivstellensatz for the
  hypercube}.
\newblock {\em Optim. Lett.}, 17(3):515--530, 2023.

\bibitem[LS26]{OverviewConvergenceRates}
Monique Laurent and Lucas Slot.
\newblock An overview of convergence rates for sum of squares hierarchies
  in polynomial optimization.
\newblock In Shin'ichi Oishi, Hisashi Okamoto, and Ken Hayami, editors, {\em
  Recent Developments in Industrial and Applied Mathematics}, pages 149--176,
  Singapore, 2026. Springer Nature Singapore.

\bibitem[Mag25]{rateprod}
Victor Magron.
\newblock Convergence rates for polynomial optimization on set products.
\newblock Preprint, {arXiv}:2505.18580 [math.{OC}] (2024), 2025.

\bibitem[Mar08]{marshall}
Murray Marshall.
\newblock {\em Positive polynomials and sums of squares}, volume 146 of {\em
  Math. Surv. Monogr.}
\newblock American Mathematical Society (AMS), 2008.

\bibitem[MED21]{magron2021exact}
Victor Magron and Mohab~Safey El~Din.
\newblock On exact {R}eznick, {H}ilbert-{A}rtin and {P}utinar's
  representations.
\newblock {\em J. Symb. Comput.}, 107:221--250, 2021.

\bibitem[MM22]{mai2022complexity}
Ngoc Hoang~Anh Mai and Victor Magron.
\newblock {On the complexity of Putinar--Vasilescu's Positivstellensatz}.
\newblock {\em J. Complex.}, 72:101663, 2022.

\bibitem[NGSA{\etalchar{+}}13]{navascues2013paradox}
Miguel Navascu{\'e}s, Artur Garc{\'\i}a-S{\'a}ez, Antonio Ac{\'\i}n, Stefano
  Pironio, and Martin~B Plenio.
\newblock A paradox in bosonic energy computations via semidefinite programming
  relaxations.
\newblock {\em New J. Phys.}, 15(2):023026, 2013.

\bibitem[NS07]{nie2007complexity}
Jiawang Nie and Markus Schweighofer.
\newblock {On the complexity of Putinar's Positivstellensatz}.
\newblock {\em J. Complex.}, 23(1):135--150, 2007.

\bibitem[Put93]{putinar1993positive}
Mihai Putinar.
\newblock Positive polynomials on compact semi-algebraic sets.
\newblock {\em Indiana Univ. Math. J.}, 42(3):969--984, 1993.

\bibitem[Rai71]{rainville}
Earl~D. Rainville.
\newblock {\em Special functions}.
\newblock {Chelsea} {Publishing} {Comp}, 1971.

\bibitem[Rob55]{robbins}
Herbert Robbins.
\newblock A remark on {Stirling}'s formula.
\newblock {\em Am. Math. Mon.}, 62:26--29, 1955.

\bibitem[Sch91]{Schmudgen1991}
Konrad Schm\"udgen.
\newblock The {K}-moment problem for compact semi-algebraic sets.
\newblock {\em Math. Ann.}, 289:203--206, 1991.

\bibitem[Sch04]{schweighofer2004complexity}
Markus Schweighofer.
\newblock On the complexity of {S}chm{\"u}dgen's {P}ositivstellensatz.
\newblock {\em J. Complex.}, 20(4):529--543, 2004.

\bibitem[Sch17]{schmuedgen:TheMomentProblem}
Konrad Schm{\"u}dgen.
\newblock {\em The Moment Problem}.
\newblock Springer, 2017.

\bibitem[SL23]{slot2023sum}
Lucas Slot and Monique Laurent.
\newblock Sum-of-squares hierarchies for binary polynomial optimization.
\newblock {\em Math. Program.}, 197(2):621--660, 2023.

\bibitem[Slo22]{slot2022sum}
Lucas Slot.
\newblock Sum-of-squares hierarchies for polynomial optimization and the
  {C}hristoffel--{D}arboux kernel.
\newblock {\em SIAM J. Optim.}, 32(4):2612--2635, 2022.

\bibitem[Sze75]{szego:ortoPoly}
Gabor Szeg{\"o}.
\newblock {\em Orthogonal polynomials. 4th ed}, volume~23 of {\em Colloq.
  Publ.}
\newblock American Mathematical Society (AMS), 1975.

\bibitem[Wak12]{waki2012generate}
Hayato Waki.
\newblock How to generate weakly infeasible semidefinite programs via
  {L}asserre’s relaxations for polynomial optimization.
\newblock {\em Optim. Lett.}, 6(8):1883--1896, 2012.

\end{thebibliography}
\newcommand{\etalchar}[1]{$^{#1}$}

\end{document}